\documentclass[english,15pt]{article}
\usepackage{hyperref}
\usepackage[T1]{fontenc}
\usepackage[latin1]{inputenc}
\usepackage{geometry}
\usepackage{float}
\usepackage{mathrsfs}
\usepackage{amsmath}
\usepackage{amssymb}
\usepackage{graphicx}
\usepackage{subcaption}

\hypersetup{
    colorlinks=true,
    linkcolor=blue,
    filecolor=magenta,      
    urlcolor=cyan,
    citecolor = blue,
}

\makeatletter

\usepackage{algorithm}
 \usepackage{algorithmicx}

\usepackage{todonotes}

\usepackage{algorithm,algpseudocode}

\usepackage{amsthm}

\usepackage{mathrsfs}

\usepackage{amsfonts}

\usepackage{epsfig}

\usepackage{bm}

\usepackage{mathrsfs}

\usepackage{enumerate}

\numberwithin{equation}{section}

\@ifundefined{definecolor}{\@ifundefined{definecolor}
 {\@ifundefined{definecolor}
 {\usepackage{color}}{}
}{}
}{}

\newtheorem{theorem}{Theorem}[section]
\newtheorem{lem}{Lemma}[section]

\newcounter{hypA}
\newenvironment{hypA}{\refstepcounter{hypA}\begin{itemize}
  \item[({\bf A\arabic{hypA}})]}{\end{itemize}}
\newcounter{hypB}

\newcounter{hypD}

\usepackage{babel}\date{}

\usepackage{babel}

\makeatother

\usepackage{babel}

\begin{document}

\begin{center}

{\Large \textbf{Parameter Estimation for Partially Observed Time-Changed SDEs}}

\vspace{0.5cm}

  KE ZHAO,  \& AJAY JASRA

{\footnotesize School of Data Science,  The Chinese University of Hong Kong,  Shenzhen,  Shenzhen, CN.}\\
{\footnotesize E-Mail:\,} \texttt{\emph{\footnotesize  kezhao@link.cuhk.edu.cn; ajayjasra@cuhk.edu.cn
}}

\end{center}

\begin{abstract}
In this paper we consider the parameter estimation problem associated to partially-observed time changed SDEs,  with observations that are given at discrete times.  In particular we consider both likelihood and Bayesian estimation.
We develop new Markov chain Monte Carlo (MCMC) algorithms which allow an unbiased score-based stochastic approximation method to provide likelihood-type parameter estimators.   We also use a variant of this MCMC algorithm to perform multilevel-based Bayesian parameter estimation.  We prove that this latter method achieves a mean square error of $\mathcal{O}(\epsilon^2)$ ($\epsilon>0$) with a cost of $\mathcal{O}(\epsilon^{-2}\log(\epsilon)^2)$.  Our methodologies are tested numerically on both simulated and real data. 
\end{abstract}

\section{Introduction}

Time-changed stochastic differential equations (SDEs) provide a flexible modeling framework for systems exhibiting irregular temporal dynamics, such as sub-diffusion, trapping effects, and heavy-tailed waiting times. In these models, the physical time is replaced by a stochastic time change, typically given by the inverse of an increasing L\'evy process, leading to dynamics that deviate significantly from classical diffusion behavior. Such constructions have been widely used to model anomalous transport phenomena in a range of applications, including physics \cite{metzler2000random,zaslavsky1994fractional}, finance \cite{magdziarz2009black,zhang2025sub}, and biological systems \cite{hairer2016averaging,hairer2018fractional}.  From an analytical perspective,  these processes are closely connected to fractional and nonlocal evolution equations, while probabilistically they often arise as scaling limits of continuous-time random walks \cite{biovcic2026sampling}. In recent years, there has been a growing body of work on numerical methods for such models, including simulation techniques for time-changed processes; see, for example, \cite{cazares2025fast,gonzalez2025fast, lv2020stochastic,biovcic2026sampling} and references therein.

In this article we consider partially observed time-changed SDEs associated to discretely observed data and in particular the estimation of static parameters.  This type of problem,  in terms of regular SDE or SDEs of McKean-Vlasov type as recieved a substanial amount of attention including, but not limited to, the works \cite{awadelkarim2024unbiased,beskos,chada,golightly-sherlock,golightly,graham,heng2024unbiased,jasra,ml_mv_sde}.  We consider both the cases of likelihood-based inference, using the score function (see also \cite{beskos}) and the Bayesian parameter estimation problem (e.g.~\cite{jasra}). The main challenges associated to parameter estimation for the afore-mentioned SDE problems are first that even if the transition densities of the process are analytically available,  the probability density of the skeleton of the SDE,  conditional upon the data is only available up-to a normalizing constant and second that the transition density is not available,  hence one has to resort to time-discretization methods such as Euler-Maruyama.
This has lead to a substantial literature that has considered how to combine conventional score-based likelihood estimation methods with multilevel Monte Carlo \cite{giles2008multilevel} and advanced simulation techniques such as particle filters and Markov chain Monte Carlo (MCMC) see e.g.~\cite{awadelkarim2024unbiased,beskos}.  
Indeed \cite{awadelkarim2024unbiased} construct a score-based approach which is unbiased,  in that the expected value of the estimator produced is equal to the true optimizer of the likelihood the latter of which is for the exact, not time-discretized model. 
The Bayesian parameter estimation problem has also lead to similar contributions such as in \cite{jasra,ml_mv_sde}.

Given the substantial literature for parameter estimation we consider how this can be developed for the case 
of partially observed time-changed SDEs.  As we explain later in the article it is not simply a matter of taking the existing methodology and applying it,  as one must be careful to tailor such ideas to the context here.  In this article we make several contributions to the literature as we now state.   Based on the Girsanov formula we provide an expression for the score function of the model and then develop new particle simulation-based methods to potentially provide unbiased estimators of the parameters.  This new algorithm is related to \cite{awadelkarim2024unbiased},  but requires several subtle modifications for the approach to work.  We note that by unbiased we mean that the expected value of the estimator is the true optimizer of the likelihood without time-discretization bias; we do not prove this result, although we expect it follows from the work of \cite{awadelkarim2024unbiased}.  We then show how the approach developed for score-based parameter estimation can be adapted to the case of Bayesian parameter estimation; this latter approach is essentially that of \cite{jasra},  except with a new MCMC method.  We prove that our estimator can achieve a mean square error (MSE) of $\mathcal{O}(\epsilon^2)$ ($\epsilon>0$) for a cost of $\mathcal{O}(\epsilon^{-2}\log(\epsilon)^2)$, which is close to optimal.  We then illustrate our two approaches on both simulated and real data.

The remainder of the paper is organized as follows.  Section \ref{sec:problem} introduces the problem formulation. Section \ref{sec:approach} details the methodology and construction of our algorithms.  Section \ref{sec:theory} presents the main theoretical results and Section \ref{sec:numerics} provides numerical experiments. Technical proofs are given in Appendix \ref{app:proof}.

\section{Problem formulation}\label{sec:problem}

\subsection{Model}

For any $\theta \in \Theta \subset \mathbb{R}^{d_\theta}$
let $(\Omega, \mathcal{F}, \mathbb{P}_{\theta})$ be a probability space. We consider a continuous-time latent process $\{X_t\}_{t \geq 0}$ evolving according to a time-changed SDE:
\begin{equation}
\label{eq:tcsde}
dX_t = a_\theta(X_t)\,dL_t + \sigma(X_t)\,dB_{L_t}, 
\quad X_0 = x_0 \in \mathbb{R}^{d_x}.
\end{equation}
In the above display,  
$\{B_t\}_{t \geq 0}$ is a standard $d_x$-dimensional Brownian motion, and let $\{D_t\}_{t \geq 0}$ be a sub-ordinator (i.e., a non-decreasing L\'evy process) independent of $B_t$.  We define the inverse sub-ordinator
$$
L_t = \inf\{ s \geq 0 : D_s > t \}, \quad t \geq 0,
$$
which is an increasing, non-Markovian process that induces random time changes. The process $\{B_{L_t}\}_{t \geq 0}$ is then a time-changed Brownian motion. Throughout we make the following assumption  which ensures a unique strong solution. 
\begin{hypA}\label{assump:coefficients}
For $\theta\in\Theta$,  let $a_{\theta} : \mathbb{R}^{d_x} \to \mathbb{R}^{d_x}$ and
$\sigma : \mathbb{R}^{d_x} \to \mathbb{R}^{d_x\times d_x}$ be measurable functions.
There exists $C<+\infty$ such that for 
$(\theta,x,y) \in \Theta\times\mathbb{R}^{2d_x}$ 
\begin{eqnarray*}
|a_{\theta}(x) - a_{\theta}(y)| + |\sigma(x) - \sigma(y)|
&\le & C|x - y|, \label{ineq:Lip1}\\
|a_{\theta}(x)| + |\sigma(x)|
&\le & C\big(1 + |x|\big), \label{ineq:Lip2}\\
|a_{\theta}(x) - a_{\theta}(x)| + |\sigma(x) - \sigma(x)|
&\le & C\big(1 + |x|\big),
\end{eqnarray*}
where $|\cdot|$ denotes the Euclidean norm
in the corresponding space.
\end{hypA}
A detailed discussion of stochastic integrals and SDEs driven by time-changed semimartingales is provided in \cite{kobayashi2011stochastic} including the existence and uniqueness of the solution.  We suppose that $X_0$ is stochastic with a Lebesgue probability density $\mu_{\theta}$.

The latent process is observed at discrete times $T\in\mathbb{N}$ unit times; this convention is only used to simplify the notation and one can easily extend the discussion to the case of (possibly irregular) observation times $0<t_1<\cdots<t_K=T$ say.
 For each $k \in \{1,\dots,T\}$, the observation $Y_k$ takes values in a measurable space $(\mathsf{Y}, \mathcal{Y})$, and its conditional distribution given the latent state $X_{k} = x$ admits a density 
$
G_\theta(x, y_k)$, 
with respect to a reference measure (usually counting or Lebesgue) on $(\mathsf{Y}, \mathcal{Y})$. We remark that our approach to parameter estimation will rely on the Girsanov and gradient (score) methods which will circumvent allowing $\sigma$ in \eqref{eq:tcsde} to depend on $\theta$.  We discuss this in Section \ref{sec:sbpe}.

\subsection{Score-Based Parameter Estimation}

The exposition in this section is similar to \cite[Section 3.1]{awadelkarim2024unbiased} with an appropriate modification of results to the context in this article. 
Let $\{X_t\}_{t \in [0,T]}$ be the solution to \eqref{eq:tcsde}. For a given test function $\varphi_\theta : \mathbb{R}^{d_x \times K} \to \mathbb{R}$ defined on a finite skeleton $\{1,2,...,T\}$, we consider
\[
\varphi_\theta(X_{1}, \dots, X_{T}).
\]
Assuming integrability, the expectation under $\mathbb{P}_\theta$ can be expressed using a reference measure $\mathbb{Q}$ via Girsanov's theorem:
\[
\mathbb{E}_\theta[\varphi_\theta(X_{1}, \dots, X_{T})]
= \mathbb{E}_{\mathbb{Q}\otimes\text{Leb}} \Bigg[
\mu_\theta(X_0)\varphi_\theta(X_{1}, \dots, X_{T})
\frac{d\mathbb{P}_\theta}{d\mathbb{Q}}
\Bigg]
\]
where the Radon--Nikodym derivative is given by \cite{zhang2025sub}
\begin{equation}
\frac{d\mathbb{P}_\theta}{d\mathbb{Q}} =
\exp \Bigg\{
- \frac{1}{2} \int_0^T \|b_\theta(X_s)\|^2 d{L_s}
+ \int_0^T b_\theta(X_s)^\top d\widetilde{W}_{L_s}
\Bigg\},
\end{equation}
\(b_{\theta}(x)=\Sigma(x)^{-1}\sigma(x)^\top a_{\theta}(x)\), \(\Sigma(x)=\sigma(x)\sigma(x)^\top\) and $\mathbb{Q}\otimes\text{Leb}$ is the product measure of $\mathbb{Q}$ (to be described) and Lebesgue measure on $x_0$.  Under $\mathbb{Q}$ the process satisfies
\[
dX_t = \sigma(X_t)d\widetilde{W}_{L_t},
\]
and \(\widetilde{W}_{L_t}\) is independent of \(W_{L_t}\) and has the same distribution as \(W_{L_t}\).
Using the observation model, we define
\[
\varphi_\theta(X_{1:T}) = \prod_{k=1}^T G_\theta(X_k, y_k),
\]
so that the marginal likelihood is
\[
p_\theta(y_{1:T}) = \mathbb{E}_\theta\Bigg[\prod_{k=1}^T G_\theta(X_{k}, y_k)\Bigg].
\]
The objective is to maximize $p_\theta(y_{1:T})$ using the score function $\nabla \log p_\theta(y_{1:T})$ a task which we assume is well-posed; as the nature of this work is primarily methodological we do not specify these standard conditions for the problem to be well-defined.

Define the function:
\begin{align}
H_\theta(\{X_{t}\}_{t\in[0,T]},\{L_{t}\}_{t\in[0,T]}) &= \nabla_\theta \log \mu_\theta(X_0)
+ \sum_{k=1}^T \nabla_\theta \log G_\theta(X_{k}, y_k) \nonumber\\
&\quad - \frac{1}{2} \int_0^T \nabla_\theta \|b_\theta(X_s)\|^2 d{L_s}
+ \int_0^T \nabla_\theta b_\theta(X_s)^\top \Sigma(X_s)^{-1}\sigma(X_s)^\top dX_s.\label{eq:noisy_score_ct}
\end{align}
In addition,  set $\psi_{\theta}(x_1,\dots,x_T)=\prod_{k=1}^T G_{\theta}(x_k,y_k)$ set the probability measure $\pi_{\theta}=\psi_{\theta}\mathbb{P}_{\theta}/\mathbb{P}_{\theta}(\psi_{\theta})$ where $\mathbb{P}_{\theta}(\psi_{\theta})=p_\theta(y_{1:T})$. 
Under standard regularity conditions, the score function admits the representation
\[
\nabla_\theta \log p_\theta(y_{1:T})
= \mathbb{E}_{\pi_\theta}\left[H_\theta(\{X_{t}\}_{t\in[0,T]},\{L_{t}\}_{t\in[0,T]})\right].
\]
Therefore we would like to find $\theta$ so that
\[
h(\theta) := \mathbb{E}_{\pi_\theta}\left[H_\theta(\{X_{t}\}_{t\in[0,T]},\{L_{t}\}_{t\in[0,T]})\right] = 0.
\]

\subsection{Bayesian Parameter Estimation}

The Bayesian approach is rather simpler as we now describe.  Let $\overline{\pi}$ be a Lebesgue probability density on the space $\Theta$ such that $\int_{\Theta}p_{\theta}(y_{1:T})\overline{\pi}(\theta)d\theta<+\infty$, with 
$d\theta$ the $d_{\theta}-$dimensional Lebesgue measure. 
Let $\varphi:\Theta\rightarrow\mathbb{R}$ be a measurable function such that 
$\int_{\Theta}\varphi(\theta)\mathbb{E}_{\theta}[\psi(X_{1:T})]\overline{\pi}(\theta)d\theta<+\infty$.
The objective in Bayesian parameter estimation is then to approximate
\begin{equation}\label{eq:bpe_def}
\pi(\varphi) := \frac{\int_{\Theta}\varphi(\theta)\mathbb{E}_{\theta}[\psi(X_{1:T})]\overline{\pi}(\theta)d\theta}
{\int_{\Theta}p_{\theta}(y_{1:T})\overline{\pi}(\theta)d\theta}
\end{equation}
that is expectations with respect to the marginal posterior on $\Theta$.  We note that in the Bayesian case $\sigma$ can depend on $\theta$ with little difficulty as the approach does not depend on the score-function which has been computed via Girsanov. 


\section{Methodology}\label{sec:approach}

\subsection{On the Inverse Sub-ordinator}

A key component of the time-changed SDE \eqref{eq:tcsde} is the inverse sub-ordinator $\{L_t\}_{t \geq 0}$. In general, simulating $L_t$ exactly is nontrivial due to its non-Markovian nature and the lack of closed-form transition densities. Nevertheless, for certain classes of sub-ordinators (e.g., stable sub-ordinators), exact or numerically exact simulation methods have been developed in the literature; see, e.g., \cite{cazares2025fast,gonzalez2025fast}.
In this work, we assume that one can generate exact samples from the finite-dimensional distributions of $\{L_t\}$ at a prescribed set of time points. We use the notation for $0\leq t_{k-1}<t_k\leq T$
\begin{equation}\label{eq:law of L}
P_{\theta,L}(l_{t_{k-1}},dl_{t_k})    
\end{equation}
for the transition kernel in the exact simulation \(L_t\).

\subsection{Time Discretization}

In practice, in order to perform any type of parameter estimation,  we will have to resort to time discretization.
In this section $\theta\in\Theta$ is fixed.
Let $\texttt{l} \in \mathbb{N}_0$ be a discretization level and define the time step $\Delta_{\texttt{l}} = 2^{-\texttt{l}}$. 
We consider the Euler-Maruyama scheme for $k\in\{0,\dots,T\Delta_{\texttt{l}}^{-1}-1\}$
\begin{equation}
\label{eq:euler}
X_{(k+1)\Delta_{\texttt{l}}}
=
X_{k\Delta_{\texttt{l}}}
+ b_\theta(X_{k\Delta_{\texttt{l}}}) \left(L_{(k+1)\Delta_{\texttt{l}}}-L_{k\Delta_{\texttt{l}}}\right)
+ \sigma(X_{(k+1)\Delta_{\texttt{l}}})
\left(W_{L_{(k+1)\Delta_{\texttt{l}}}} - W_{L_{k\Delta_{\texttt{l}}}}\right),
\end{equation}
where we recall that $X_0$ has distribution $\mu_{\theta}(x_0)dx_0$ and the inverse sub-ordinator follows the transition law \eqref{eq:law of L} with $L_0=0$.

To simplify the notation set for $(k,\texttt{l})\in\{1,\dots,T\}\times\mathbb{N}_{0}$
$$
U_k^{\texttt{l}} = \left((X_{(k-1)+\Delta_{\texttt{l}}},L_{(k-1)+\Delta_{\texttt{l}}}), \dots, (X_k,L_{k})\right)
$$
with $U_0^l = (X_0,L_0)$.
We define the posterior smoother at level $\texttt{l}\in\mathbb{N}_0$
\begin{equation}\label{eq:post_smooth}
\pi_\theta^{\texttt{l}}(du_{0:T}^{\texttt{l}})
\propto
\left\{\prod_{k=1}^{T}
\overline{M}_{\theta,\texttt{l}}(u_{k-1}^{\texttt{l}},du_k^{\texttt{l}})
 G_\theta(x_{k}, y_k)\right\}\overline{\mu}_\theta(du_0)
\end{equation}
where $\overline{\mu}_\theta(du_0)=\mu_{\theta}(dx_0)\delta_{\{0\}}(dl_0)$,  $\delta_{\{0\}}(dl_0)$ the dirac measure on $0$
$$
\overline{M}_{\theta,\texttt{l}}(u_{k-1}^{\texttt{l}},du_k^{\texttt{l}}) = \prod_{s=1}^{\Delta_{\texttt{l}}^{-1}}
M_{\theta,\texttt{l}}^{
}\left(v_{k-1+s\Delta_{\texttt{l}}}, dx_{(k-1)+s\Delta_{\texttt{l}}}\right) P_{\theta,L}(l_{(k-1)+(s-1)\Delta_{\texttt{l}}},dl_{(k-1)+s\Delta_{\texttt{l}}})
$$
$v_{k-1+s\Delta_{\texttt{l}}}=(l_{(k-1)+(s-1)\Delta_{\texttt{l}}},l_{(k-1)+s\Delta_{\texttt{l}}},x_{(k-1)+(s-1)\Delta_{\texttt{l}}})$ and
$M_{\theta,\texttt{l}}$ is the Markov kernel induced by the Euler scheme \eqref{eq:euler}.

We will also need to define a joint probability measure on $(u_{0:T}^{\texttt{l}},u_{0:T}^{\texttt{l}-1})$, $\texttt{l}\in\mathbb{N}$ which is related to \eqref{eq:post_smooth}; this latter probability will be integral in our algorithms. 
To that end, consider an additional $\theta'\in\Theta$ and the well known synchronous coupling of Euler-Maruyama discretizations for $(k,k')\in\{0,\dots,T\Delta_{\texttt{l}}^{-1}-1\}\times\{0,\dots,T\Delta_{\texttt{l-1}}^{-1}-1\}$:
\begin{eqnarray*}
X_{(k+1)\Delta_{\texttt{l}}}
& = &
X_{k\Delta_{\texttt{l}}}
+ b_\theta(X_{k\Delta_{\texttt{l}}}) \left(L_{(k+1)\Delta_{\texttt{l}}}-L_{k\Delta_{\texttt{l}}}\right)
+ \sigma(X_{(k+1)\Delta_{\texttt{l}}})
\left(W_{L_{(k+1)\Delta_{\texttt{l}}}} - W_{L_{k\Delta_{\texttt{l}}}}\right)\\
X_{(k'+1)\Delta_{\texttt{l}-1}}
& = &
X_{k'\Delta_{\texttt{l-1}}}
+ b_{\theta'}(X_{k'\Delta_{\texttt{l}-1}}) \left(L_{(k'+1)\Delta_{\texttt{l}-1}}-L_{k'\Delta_{\texttt{l}-1}}\right)
+ \sigma(X_{(k'+1)\Delta_{\texttt{l}-1}})\times \\ & &
\left(W_{L_{(k'+1)\Delta_{\texttt{l}-1}}} - W_{L_{k'\Delta_{\texttt{l}-1}}}\right).
\end{eqnarray*}
In the above display the inverse sub-ordinator and the Brownian motion are the same across the two levels $\texttt{l}$ and $\texttt{l}-1$,  in the sense that concatenating the fine (level $\texttt{l}$) increments will yield the coarser (level $\texttt{l}-1$) increments. In practice,  this is easily achieved by sampling the inverse sub-ordinator and the Brownian motion on the fine grid and summing increments to obtain the coarse increments.  Setting $\mathbf{u}_k^{\texttt{l}}=(u_k^{\texttt{l}},u_k^{\texttt{l}-1})$ for $k\in\{0,\dots,T\}$ we introduce
$$
\widetilde{M}_{\theta,\theta',\texttt{l}}(\mathbf{u}_{k-1}^{\texttt{l}},d\mathbf{u}_k^{\texttt{l}})  =  
$$
$$
\prod_{s=1}^{\Delta_{\texttt{l}-1}^{-1}}\Bigg\{
\check{M}_{\theta,\theta',\texttt{l}}\left(\left(v_{k-1+(2s-1)\Delta_{\texttt{l}}}^{\texttt{l}},l_{k-1+2s\Delta_{\texttt{l}}}^{\texttt{l}}\right), d\left(
x_{k-1+(2s-1)\Delta_{\texttt{l}}}^{\texttt{l}},x_{k-1+2s\Delta_{\texttt{l}}}^{\texttt{l}}
x_{(k-1)+s\Delta_{\texttt{l}-1}}^{\texttt{l}-1}\right)\right)\otimes
$$
$$
\delta_{\{l_{k-1+2s\Delta_{\texttt{l}}}^{\texttt{l}}\}}(dl_{k-1+s\Delta_{\texttt{l}-1}}^{\texttt{l}-1})\Bigg\}\otimes 
\prod_{s=1}^{\Delta_{\texttt{l}}^{-1}} P_{\theta,L}(l_{(k-1)+(s-1)\Delta_{\texttt{l}}}^{\texttt{l}},dl_{(k-1)+s\Delta_{\texttt{l}}}^{\texttt{l}})
$$
where $\check{M}_{\theta,\theta',\texttt{l}}$ is the transition kernel that is induced by the synchronous coupling of the Euler-Maruyama discretication of the time changed SDE, we use the super-scripts $\texttt{l}$ and $\texttt{l}-1$ to distinguish the skeleton of the level $\texttt{l}$ and $\texttt{l}-1$ paths of the SDE and it should be noted that the inverse sub-ordinator is only sampled at the fine level and concatened at the coarse level to produce the path of inverse sub-ordinators at the coarse level.  For the kernel $\widetilde{M}_{\theta,\theta',\texttt{l}}$ if $\theta=\theta'$ we simply write $\widetilde{M}_{\theta,\texttt{l}}$.

Let $\overline{\mu}_{\theta,\theta'}$ be any coupling of $(\mu_{\theta}(x_0)dx_0,\mu_{\theta'}(x_0')dx_0')$ 
and set $\check{\mu}_{\theta,\theta'}(d\mathbf{u}_0^{\texttt{l}})=
\overline{\mu}_{\theta,\theta'}(d(x_0^{\texttt{l}},x_0^{\texttt{l}-1}))\otimes \delta_{\{0\}}(dl_0^{\texttt{l}})\otimes\delta_{\{0\}}(dl_0^{\texttt{l}-1})$,
then we introduce the following probability measure for $\textrm{l}\in\mathbb{N}$:
\begin{equation}\label{eq:post_smooth_coup}
\check{\pi}_{\theta,\theta'}^{\texttt{l}}\left(d\mathbf{u}_{0:T}^{\texttt{l}}\right)
\propto
\left\{\prod_{k=1}^{T} \widetilde{M}_{\theta,\theta',\texttt{l}}(\mathbf{u}_{k-1}^{\texttt{l}},d\mathbf{u}_k^{\texttt{l}})
 \check{G}_{\theta,\theta'}(x_{k},^{\texttt{l}} x_k^{\texttt{l}-1},y_k)\right\}\check{\mu}_{\theta,\theta'}(d\mathbf{u}_0^{\texttt{l}})
\end{equation}
where $\check{G}_{\theta,\theta'}(x_{k},^{\texttt{l}} x_k^{\texttt{l}-1},y_k) = G_{\theta}(x_k^{\texttt{l}},y_k)+G_{\theta'}(x_k^{\texttt{l}-1},y_k)$.  As above if $\theta=\theta'$ we write $\check{\pi}_{\theta}^{\texttt{l}}$ and 
$\check{G}_{\theta}$.  The two target probabilities \eqref{eq:post_smooth} and \eqref{eq:post_smooth_coup} will be fundamental in our two parameter estimation methodologies.  We now detail how one can construct Markov kernels which can be used to generate approximate samples from  \eqref{eq:post_smooth} and \eqref{eq:post_smooth_coup}.

\subsection{Sampling Algorithms}

We now describe a conditional particle filter (CPF) targeting $\pi_\theta^{\texttt{l}}$ in \eqref{eq:post_smooth}. This construction is essentially that in \cite{andrieu2010particle} which is a Markov kernel that allows one to deliver approximate samples from $\pi_\theta^{\texttt{l}}$.  The simulation of one step is given in Algorithm \ref{alg:cpf}.
Note that in line \ref{alg:ini} i.i.d.~stands for independent and identically distributed.   In line \ref{alg:res} (see also line \ref{alg:grand_res}) the notation $\mathcal{C}\text{at}(\omega_k^1,\dots,\omega_k^N)$ means the categorical probability distribution on $\{1,\dots,N\}$ with probabilities $(\omega_k^1,\dots,\omega_k^N)$.

\begin{algorithm}[h]
\caption{Conditional particle filter at level $\texttt{l}\in\mathbb{N}$ for $\theta\in\Theta$ given.}
\begin{algorithmic}[1]
\State Input reference trajectory $u'_{0:T}$.\label{alg:cpf_input}
\State Initialize particles $U_0^i \stackrel{\text{i.i.d.}}{\sim} \overline{\mu}_\theta$, $i\in\{1,\dots,N-1\}$,  set $U_0^N = u'_0$. \label{alg:ini}

\For{$k\in\{1,\dots,T\}$}\label{alg:cpf_for}
    \State \textbf{Sampling:}
    \For{$i\in\{1,\dots,N-1\}$}
        \State Sample $U_k^i|u_{k-1}^{i} \sim \overline{M}_{\theta,\texttt{l}}(u_{k-1}^{i}, \cdot)$.
    \EndFor
    \State Set $U_k^N = u'_k$.\label{alg:cpf_update}

    \State \textbf{Weighting:}  For $i\in\{1,\dots,N\}$ set
    \[
    \omega_k^i = \frac{G_\theta(x_k^i, y_k)}{\sum_{j=1}^{N} G_\theta(x_k^j, y_k)}.
    \]

    \State \textbf{Resampling:} If $k\leq T-1$,  for $i\in\{1,\dots,N-1\}$ sample $A_k^i$ from $\mathcal{C}\text{at}(\omega_k^1,\dots,\omega_k^N)$ and set $u_{0:k}^i=u_{0:k}^{A_k^{i}}$.\label{alg:res}
\EndFor

\State Sample trajectory index $I$ using $\mathcal{C}\text{at}(\omega_T^1,\dots,\omega_T^N)$ and return $u_{0:T}^I$.\label{alg:grand_res}
\end{algorithmic}
\label{alg:cpf}
\end{algorithm}

%
%
%
%
%
%

In order to obtain approximate samples from \eqref{eq:post_smooth_coup} we consider a conditional delta particle filter in Algorithm \ref{alg:cond_delta_pf}.  The delta particle filter was developed in \cite{chada,jasra} and is essentially like Algorithm \ref{alg:cpf} adapted to the product space nature of the target \eqref{eq:post_smooth_coup}.

\begin{algorithm}[h]
\caption{Conditional Delta Particle Filter at level $\texttt{l} \in \mathbb{N}$ for $(\theta,\theta')\in\Theta^2$ given.}
\begin{algorithmic}[1]

\State \textbf{Input:} reference trajectory $\mathbf{u}_{0:T}^{\prime,\texttt{l}}$

\State Sample $(U_0^{i,\texttt{l}},U_{0}^{i,\texttt{l}-1}) \stackrel{\text{i.i.d.}}{\sim} \check{\mu}_{\theta,\theta'}$,  for $i\in\{1,\dots,N-1\}$.
\State Set $U_0^{N,\texttt{l}} = u_0^{\prime,\texttt{l}}$, \quad $U_0^{N,\texttt{l}-1} = U_0^{\prime,\texttt{l}-1}$, $k=1$.

\For{$k\in\{1,\dots,T\}$}

    \State \textbf{Sampling step:}

    \For{$i\in\{1, \dots, N-1\}$}
        \State Sample $\mathbf{U}_k^{i,\texttt{l}}|\mathbf{u}_{k-1}^{i,\texttt{l}}
\sim \widetilde{M}_{\theta,\theta',\texttt{l}}(\mathbf{u}_{k-1}^{i,\texttt{l}},\cdot)$.
    \EndFor

    \State Set $U_k^{N,\texttt{l}} = u_k^{\prime,\texttt{l}}$,  $U_k^{N,\texttt{l}-1} = u_k^{\prime,\texttt{l}-1}$.



\State \textbf{Weighting:}  For $i\in\{1,\dots,N\}$ set
    \[
    	\check{\omega}_k^i = \frac{\check{G}_{\theta,\theta'}(x_k^{i,\texttt{l}},x_k^{i,\texttt{l}-1},y_k)}{\sum_{j=1}^N
\check{G}_{\theta,\theta'}(x_k^{j,\texttt{l}},x_k^{j,\texttt{l}-1},y_k)}.
    \]
\State \textbf{Resampling:} If $k\leq T-1$,  for $i\in\{1,\dots,N-1\}$ sample $A_k^i$ from $\mathcal{C}\text{at}(\check{\omega}_k^1,\dots,\check{\omega}_k^N)$ and set $\mathbf{u}_{0:k}^{i,\texttt{l}}=\mathbf{u}_{0:k}^{A_k^{i},\texttt{l}}$.

%
    \EndFor
\State Sample trajectory index $I$ using $\mathcal{C}\text{at}(\check{\omega}_T^1,\dots,\check{\omega}_T^N)$ and return $\mathbf{u}_{0:T}^{I,\texttt{l}}$.
%
%
%

\end{algorithmic}
\label{alg:cond_delta_pf}
\end{algorithm}


\subsection{Score-Based Parameter Estimation}\label{sec:sbpe}

We now show how the elements described so far can be used for unbiased parameter estimation using the ideas in \cite{awadelkarim2024unbiased}; it is remarked however,  that the method is not identical and we return to this point later in this section.

We begin by specifing the time discretized version of $H_{\theta}$ in \eqref{eq:noisy_score_ct}.  We have for $\texttt{l}\in\mathbb{N}_0$:
\begin{align*}
H_{\theta}^{\texttt{l}}(U_{0:T}^{\textrm{l}}) &= \nabla_\theta \log \mu_\theta(X_0^{\texttt{l}})
+ \sum_{k=1}^T \nabla_\theta \log G_\theta(X_{k}^{\texttt{l}}, y_k) 
 - \frac{1}{2}\sum_{k=0}^{T\Delta_{\texttt{l}}^{-1}-1} \Bigg\{\nabla_\theta \|b_\theta(X_{k\Delta_{\texttt{l}}}^{\texttt{l}})\|^2\{
L_{(k+1)\Delta_{\texttt{l}}} - L_{k\Delta_{\texttt{l}}} 
\} \\ & 
-2 \nabla_\theta b_\theta(X_{k\Delta_{\texttt{l}}}^{\texttt{l}})^\top \Sigma(X_{k\Delta_{\texttt{l}}}^{\texttt{l}})^{-1}\sigma(X_{k\Delta_{\texttt{l}}}^{\texttt{l}})^\top \{X_{(k+1)\Delta_{\texttt{l}}}^{\texttt{l}}-X_{k\Delta_{\texttt{l}}}^{\texttt{l}}\}\Bigg\}.
\end{align*}
Set 
$$
h^{\texttt{l}}(\theta) = \mathbb{E}_{\pi_{\theta}^{\texttt{l}}}\left[H_{\theta}^{\texttt{l}}(U_{0:T}^{\textrm{l}})\right]
$$
where $\mathbb{E}_{\pi_{\theta}^{\texttt{l}}}$ is an expectation w.r.t.~the probability in \eqref{eq:post_smooth}.
Let $\theta^{\texttt{l}}_\star$ denote the root of $h^{\texttt{l}}(\theta)=0$ and $\theta_\star$ the root of $h(\theta)=0$.  We assume $\lim_{\texttt{l}\to\infty}\theta^{\texttt{l}}_\star=\theta_\star$, and, for simplicity, that both roots are unique---though it suffices to assume the existence of a consistent family of solutions.

The objective is to introduce an algorithm that can (potentially) recover $\theta_{\star}$ in expectation whilst only working with time discretizations and this will work by randomizing over $\texttt{l}$ and the number of iterations of a stochastic approximation (SA) method. 
Introduce a positive probability mass function $\mathbb{P}_{\texttt{L}}(\texttt{l})$ on $\mathbb{N}_0$ for sampling the approximation level.  Let $N_{\texttt{p}}=2^{\texttt{p}}$ be the number of iterations of an SA method and consider a positive probability mass function $\mathbb{P}_{\texttt{P}}(\texttt{p})$ on $\mathbb{N}_0$ for sampling the number of iterations.  We denote the Markov kernel described in Algorithm \ref{alg:cpf} as $K_{\theta,l}$ and the Markov kernel in Algorithm \ref{alg:cond_delta_pf} as $\check{K}_{\theta,\theta',l}$.  We use the notation
\begin{eqnarray*}
\nu_{\theta,\texttt{l}}(du_{0:T}^{\texttt{l}}) & = &  
\left\{\prod_{k=1}^{T}\overline{M}_{\theta,\texttt{l}}(u_{k-1}^{\texttt{l}},du_k^{\texttt{l}})\right\}\overline{\mu}_\theta(du_0^{\texttt{l}})\quad \texttt{l}\in\mathbb{N}_0 \\
\check{\nu}_{\theta,\theta',\texttt{l}}(d\mathbf{u}_{0:T}^{\texttt{l}}) & = &
 \left\{\prod_{k=1}^{T} \widetilde{M}_{\theta,\theta',\texttt{l}}(\mathbf{u}_{k-1}^{\texttt{l}},d\mathbf{u}_k^{\texttt{l}})\right\}\check{\mu}_{\theta,\theta'}(d\mathbf{u}_0^{\texttt{l}})\quad \texttt{l}\in\mathbb{N}.
\end{eqnarray*}
If $\theta=\theta'$ we write $\check{\nu}_{\theta,\texttt{l}}$ instead of $\check{\nu}_{\theta,\theta,\texttt{l}}$.
We set for any $\texttt{l}\in\mathbb{N}$
\begin{eqnarray*}
\mathcal{G}_{\theta,\theta',\text{f}}(\mathbf{u}_{0:T}^{\texttt{l}}) & = & \prod_{k=1}^T \frac{G_{\theta}(x_k^{\texttt{l}},y_k)}{\check{G}_{\theta,\theta'}(x_k^{\texttt{l}},x_k^{\texttt{l}-1},y_k)} \\
\mathcal{G}_{\theta,\theta',\text{c}}(\mathbf{u}_{0:T}^{\texttt{l}}) & = & \prod_{k=1}^T \frac{G_{\theta'}(x_k^{\texttt{l}-1},y_k)}{\check{G}_{\theta,\theta'}(x_k^{\texttt{l}},x_k^{\texttt{l}-1},y_k)}
\end{eqnarray*}
and if $\theta=\theta'$ we write $\mathcal{G}_{\theta,\text{f}}$ and $\mathcal{G}_{\theta,\text{c}}$.
Finally for any $\texttt{l}\in\mathbb{N}$
\begin{eqnarray*}
\mathcal{H}_{\theta,\theta',\text{f}}^{\texttt{l}}(\bar{\mathbf{u}}_{0:T}^{\texttt{l}},\mathbf{u}_{0:T}^{\texttt{l}}) & = &
\frac{\mathcal{G}_{\theta,\theta',\text{f}}(\bar{\mathbf{u}}_{0:T}^{\texttt{l}})H^{\texttt{l}}_{\theta}(\bar{u}_{0:T}^{\texttt{l}})
+
\mathcal{G}_{\theta,\theta',\text{f}}(\mathbf{u}_{0:T}^{\texttt{l}})H^{\texttt{l}}_{\theta}(u_{0:T}^{\texttt{l}})
}{\mathcal{G}_{\theta,\theta',\text{f}}(\bar{\mathbf{u}}_{0:T}^{\texttt{l}})+\mathcal{G}_{\theta,\theta',\text{f}}(\mathbf{u}_{0:T}^{\texttt{l}})}  \\
\mathcal{H}_{\theta,\theta',\text{c}}^{\texttt{l-1}}(\bar{\mathbf{u}}_{0:T}^{\texttt{l}},\mathbf{u}_{0:T}^{\texttt{l}}) & = &
\frac{\mathcal{G}_{\theta,\theta',\text{c}}(\bar{\mathbf{u}}_{0:T}^{\texttt{l}})H^{\texttt{l}-1}_{\theta'}(\bar{u}_{0:T}^{\texttt{l}-1})
+
\mathcal{G}_{\theta,\theta',\text{c}}(\mathbf{u}_{0:T}^{\texttt{l}})H^{\texttt{l}-1}_{\theta'}(u_{0:T}^{\texttt{l}-1})
}{\mathcal{G}_{\theta,\theta',\text{c}}(\bar{\mathbf{u}}_{0:T}^{\texttt{l}})+\mathcal{G}_{\theta,\theta',\text{c}}(\mathbf{u}_{0:T}^{\texttt{l}})}.
\end{eqnarray*}
We are now in a position to give our method which is in Algorithm \ref{alg:umsa}.  Note in this algorithm we use a step-size sequence $\gamma_n$,  $n\in\mathbb{N}_0$ of non-negative numbers,  such that $\sum_{n\in\mathbb{N}_0} \gamma_n=\infty$ and $\sum_{n\in\mathbb{N}_0} \gamma_n^2<+\infty$ as is standard in SA methods.

%

\begin{algorithm}[h]
\caption{Unbiased Markovian stochastic approximation (UMSA)}
\begin{algorithmic}[1]
\State Sample $\texttt{l} \sim \mathbb{P}_{\texttt{L}}$, $\texttt{p} \sim \mathbb{P}_{\texttt{P}}$

\If{$\texttt{l}=0$}
    \State Initialize $\theta_0^{\texttt{l}}\in\Theta$, sample $U_{0:T}^{0,\texttt{l}} \sim \nu_{\theta_0^{\texttt{l}},\texttt{l}}$.
    \For{$n\in\{1,\dots,N_{\texttt{p}}\}$}
        \State Sample $U_{0:T}^{n,\texttt{l}}|u_{0:T}^{n-1,\texttt{l}} \sim \overline M_{\theta_{n-1}^{\texttt{l}},\texttt{l}}(u_{0:T}^{n-1,\texttt{l}}, \cdot)$.
        \State $\theta_n^{l} = \theta_{n-1}^{l} + \gamma_n H_{\theta_{n-1}^{\texttt{l}}}^{\texttt{l}}(u_{0:T}^{n,\texttt{l}})$.
    \EndFor
    \If{$\texttt{p}=0$}
        \State Return $\widehat{\theta} = \dfrac{\theta_{N_{\texttt{p}}}^{\texttt{l}}}{\mathbb{P}_{\texttt{P}}(\texttt{p})\mathbb{P}_{\texttt{L}}(\texttt{l})}$.
    \Else
        \State Return $\widehat{\theta} = 
\dfrac{\theta^{\texttt{l}}_{N_{\texttt{p}}}-\theta^{\texttt{l}}_{N_{\texttt{p}-1}}}
{\mathbb{P}_{\texttt{P}}(\texttt{p})\mathbb{P}_{\texttt{L}}(\texttt{l})}$.
    \EndIf

\Else
    \State Initialize $\theta_0^{\texttt{l}}=\theta_0^{\texttt{l}-1}\in\Theta$, sample $\mathbf{U}_{0:T}^{0,\texttt{l}}\sim\check{\nu}_{\theta_0^{\texttt{l}},\texttt{l}}$.
    \For{$n\in\{1,\dots,N_{\texttt{p}}\}$}
        \State Sample $\bar{\mathbf{U}}_{0:T}^{n,\texttt{l}}|\mathbf{u}_{0:T}^{n-1,\texttt{l}} \sim 
\widetilde{M}_{\theta_{n-1}^{\texttt{l}},\theta_{n-1}^{\texttt{l-1}},\texttt{l}}(\mathbf{u}_{0:T}^{n-1,\texttt{l}}, \cdot)$,
$\mathbf{U}_{0:T}^{n,\texttt{l}}|\bar{\mathbf{u}}_{0:T}^{n,\texttt{l}} \sim 
\widetilde{M}_{\theta_{n-1}^{\texttt{l}},\theta_{n-1}^{\texttt{l-1}},\texttt{l}}(\bar{\mathbf{u}}_{0:T}^{n,\texttt{l}}, \cdot)$. \label{alg:doub_samp}
        \State $\theta_n^{l} = \theta_{n-1}^{l} + \gamma_n \mathcal{H}_{\theta_{n-1}^{\texttt{l}},\theta_{n-1}^{\texttt{l}-1},\text{f}}^{\texttt{l}}(\bar{\mathbf{u}}_{0:T}^{n,\texttt{l}},\mathbf{u}_{0:T}^{n,\texttt{l}})
$. \label{alg:doub_samp1}
        \State $\theta_n^{l-1} = \theta_{n-1}^{l-1} + \gamma_n \mathcal{H}_{\theta_{n-1}^{\texttt{l}},\theta_{n-1}^{\texttt{l}-1},\text{c}}^{\texttt{l-1}}(\bar{\mathbf{u}}_{0:T}^{n,\texttt{l}},\mathbf{u}_{0:T}^{n,\texttt{l}})$.\label{alg:doub_samp2}
    \EndFor
    \If{$\texttt{p}=0$}
        \State Return $\widehat{\theta} =
        \dfrac{\theta_{N_{\texttt{p}}}^{\texttt{l}} - \theta_{N_{\texttt{p}}}^{\texttt{l}-1}}{\mathbb{P}_{\texttt{P}}(\texttt{p})\mathbb{P}_{\texttt{L}}(\texttt{l})}.$
    \Else
        \State Return
        \[
        \widehat{\theta} =
        \frac{
        \big(\theta_{N_{\texttt{p}}}^{\texttt{l}} - \theta_{N_{\texttt{p}}}^{\texttt{l}-1}\big)
        -
    \big(\theta_{N_{\texttt{p}-1}}^{\texttt{l}} - \theta_{N_{\texttt{p}-1}}^{\texttt{l}-1}\big)
        }{\mathbb{P}_{\texttt{P}}(\texttt{p})\mathbb{P}_{\texttt{L}}(\texttt{l})}.
        \]
    \EndIf
\EndIf
\end{algorithmic}
\label{alg:umsa}
\end{algorithm}

Algorithm \ref{alg:umsa} is similar to \cite{awadelkarim2024unbiased} but there is a subtle and critical difference.
In the work of \cite{awadelkarim2024unbiased} the authors used what is called the coupled conditional particle filter (CCPF) (a Markov kernel) in place of the kernel $\widetilde{M}_{\theta,\theta',l}$ that we use and that former kernel has as its invariant measure the correct marginals of the type  \eqref{eq:post_smooth} which in fact would be preferable in our case.  The price to pay in \cite{awadelkarim2024unbiased} is that the CCPF often performs worse than a delta particle filter due to the resampling operation and worse in-terms of the closeness (in some sense) of sample paths of the discretized diffuion as a function of $l$; this latter closeness is critical to obtaining finite variance estimators.  In our context using a CCPF is potentially disasterous in terms of closeness of the paths of the diffusion due to the presence of the inverse subordinator.  In effect what would happen is that paths of the inverse subordinator at different levels would get mixed up,  in terms of the particle counter and this is what would lose the closeness in the SDE paths that we would need; this has been confirmed in several empirical studies.  The conditional delta particle filter $\widetilde{M}_{\theta,\theta',l}$ does not have this problem,  but will provide approximate samples from the target probability \eqref{eq:post_smooth_coup} which does not have as its marginals probabilities such as \eqref{eq:post_smooth} and as such this is why in line \ref{alg:doub_samp} of Algorithm \ref{alg:umsa} one must produce two samples and use the weighted average type `gradients' $\mathcal{H}_{\theta,\theta',\text{f}}^{\texttt{l}}$ and $\mathcal{H}_{\theta,\theta',\text{c}}^{\texttt{l}}$ in lines \ref{alg:doub_samp1} and \ref{alg:doub_samp2}
of Algorithm \ref{alg:umsa}.  Note that producing only one sample $\mathbf{u}_{0:T}^{\texttt{l}}$ would lead to the \emph{wrong} estimators being produced which is why two samples are used.  We remark that this type of weighted SA method is different from the norm,  but at least in numerical simulations this does not lead to any noticable issue; a thorough theoretical study is warranted however; we expect that one can apply the theory of \cite{awadelkarim2024unbiased} appropriately adapted to this case.

The methodology presented here does not allow $\sigma$ to depend on $\theta$ due to the application of the Girsanov formula.  It is well-known that the latter can be a little too restrictive for score-computation and there exist several alternatives,  including using diffusion bridges such as in \cite{aj2025,beskos,non_synch};  some additional discussion can be found in \cite{heng2024unbiased}.

%
%

%
\subsection{Bayesian Estimation}

To give our approach we set for $\texttt{l}\in\mathbb{N}_0$
$$
\pi^{\texttt{l}}(d(\theta,u_{0:T}^{\texttt{l}})) \propto 
\left(\left\{\prod_{k=1}^{T}
\overline{M}_{\theta,\texttt{l}}(u_{k-1}^{\texttt{l}},du_k^{\texttt{l}})
 G_\theta(x_{k}, y_k)\right\}\overline{\mu}_\theta(du_0)
\right)
\overline{\pi}(\theta)d\theta
$$
and for $\texttt{l}\in\mathbb{N}$
$$
\check{\pi}^{\texttt{l}}(d(\theta,\mathbf{u}_{0:T}^{\texttt{l}})) \propto 
\left(
\left\{\prod_{k=1}^{T} \widetilde{M}_{\theta,\texttt{l}}(\mathbf{u}_{k-1}^{\texttt{l}},d\mathbf{u}_k^{\texttt{l}})
 \check{G}_{\theta}(x_{k},^{\texttt{l}} x_k^{\texttt{l}-1},y_k)\right\}\check{\mu}_{\theta}(d\mathbf{u}_0^{\texttt{l}})
\right)\overline{\pi}(\theta)d\theta.
$$
We denote by $K_{\texttt{l}}$ (resp.~$\check{K}_{\texttt{l}}$) any Markov kernel that can produce approximate samples from $\pi^{\texttt{l}}$ (resp.~$\check{\pi}^{\texttt{l}}$).
One example is the particle marginal Metropolis-Hastings method in \cite{andrieu2010particle} which could be based on (the unconditional version) Algorithm \ref{alg:cpf} (resp.~Algorithm \ref{alg:cond_delta_pf}).  
The unconditional (standard) versions of Algorithms \ref{alg:cpf} and \ref{alg:cond_delta_pf} do not need to input a reference trajectory (e.g.~line \ref{alg:cpf_input} of Algorithm \ref{alg:cpf}) and loop over $N$ particles not $N-1$ (e.g.~line \ref{alg:cpf_for} of Algorithm \ref{alg:cpf}) and of course do not update the $N^{\text{th}}$ particle with the reference trajectory (e.g.~line \ref{alg:cpf_update} of Algorithm \ref{alg:cpf}).  Markov kernels of the type of
$K_{\textrm{l}}$ and $\check{K}_{\texttt{l}}$ in the manner just described have been well-documented in the literature such as in \cite[Page A892,  Par 2]{jasra} and so a full description is omitted.
The Bayesian method is essentially that which has appeared in \cite{jasra} (multilevel MCMC) with an appropriate modification to this context and is presented in Algorithm \ref{alg:bayes}.

\begin{algorithm}[h]
\caption{Bayesian Multilevel Parameter Estimation.}
\begin{algorithmic}[1]
\For{$\texttt{l}\in\{0,\dots,\texttt{L}\}$}

\If{$\texttt{l}=0$}\label{alg:bpe1}
\State Set $\theta_0^{\texttt{l}}\in\Theta$ and sample $U_{0:T}^{0,\texttt{l}}\sim\nu_{\theta_0^{\texttt{l}},\texttt{l}}$.
\For{$n\in\{1,\dots,N_0\}$}
\State Sample 
$$
(\theta_n^{\texttt{l}},U_{0:T}^{n,\texttt{l}})|(\theta_{n-1}^{\texttt{l}},u_{0:T}^{n-1,\texttt{l}})\sim
K_{\textrm{l}}\left((\theta_{n-1}^{\texttt{l}},u_{0:T}^{n-1,\texttt{l}}),\cdot\right).
$$
\EndFor
\Else\label{alg:bpe2}
\State Set $\theta_0^{\texttt{l}}\in\Theta$ and sample $\mathbf{U}_{0:T}^{0,\texttt{l}}\sim\check{\nu}_{\theta_0^{\texttt{l}},\texttt{l}}$.
\For{$n\in\{1,\dots,N_{\texttt{l}}\}$}
\State Sample 
$$
(\theta_n^{\texttt{l}},\mathbf{U}_{0:T}^{n,\texttt{l}})|(\theta_{n-1}^{\texttt{l}},\mathbf{u}_{0:T}^{n-1,\texttt{l}})\sim
\check{K}_{\textrm{l}}\left((\theta_{n-1}^{\texttt{l}},\mathbf{u}_{0:T}^{n-1,\texttt{l}}),\cdot\right).
$$
\EndFor
    \EndIf
\EndFor
\end{algorithmic}
\label{alg:bayes}
\end{algorithm}

In order to estimate terms of the type
$$
\mathbb{E}_{\pi^{\texttt{L}}}[\varphi(\theta)]
$$
with $\varphi:\Theta\rightarrow\mathbb{R}$ being integrable w.r.t.~each of the $\pi^{0},\dots,\pi^{\texttt{L}}$
one can use the following estimator based on the output of Algorithm \ref{alg:bayes}:
\begin{equation}\label{eq:ml_bayes}
\widehat{\mathbb{E}_{\pi^{\texttt{L}}}[\varphi(\theta)]} = \frac{1}{N_0+1}\sum_{n=0}^{N_0}\varphi(\theta_n^0)
+\sum_{\texttt{l}=1}^{\texttt{L}}\Bigg\{
\frac{\sum_{n=0}^{N_{\texttt{l}}}
\varphi(\theta_n^{\texttt{l}})\mathcal{G}_{\theta_n^{\texttt{l}},\text{f}}(\mathbf{u}^{n,\texttt{l}}_{0:T})
}{\sum_{n=0}^{N_{\texttt{l}}}\mathcal{G}_{\theta_n^{\texttt{l}},\text{f}}(\mathbf{u}^{n,\texttt{l}}_{0:T})} - 
\frac{\sum_{n=0}^{N_{\texttt{l}}}
\varphi(\theta_n^{\texttt{l}})\mathcal{G}_{\theta_n^{\texttt{l}},\text{c}}(\mathbf{u}^{n,\texttt{l}}_{0:T})
}{\sum_{n=0}^{N_{\texttt{l}}}\mathcal{G}_{\theta_n^{\texttt{l}},\text{c}}(\mathbf{u}^{n,\texttt{l}}_{0:T})}
\Bigg\}.
\end{equation}
Note that,  in a similar manner one can also consider estimation of the states $X_0,\dots,X_T$.
The justification of \eqref{eq:ml_bayes} is as follows.  The first term on the right hand side follows as we are running an MCMC chain to target $\pi^0$ (see line \ref{alg:bpe1} of Algorithm \ref{alg:bayes}) and the summation over $\texttt{l}$ uses the collection of independent MCMC chains (see line \ref{alg:bpe2} of Algorithm \ref{alg:bayes}) that
have been generated to approximate $\check{\pi}^{\texttt{l}}$. The ratio of averages is then basically an approximation of $\mathbb{E}_{\pi^{\texttt{l}}}[\varphi(\theta)]-\mathbb{E}_{\pi^{\texttt{l}-1}}[\varphi(\theta)]$
through the same type of change of measure formula that is in \cite[eq.~(9)]{jasra}.  In the next section we explain how to choose $\texttt{L}$ and the number of iterations $N_0,\dots,N_{\texttt{L}}$.

\section{Theoretical Result}\label{sec:theory}

In order to prove our main result,  we make an additional assumption (A\ref{ass:add}) which is given in Appendix \ref{app:proof} and is discussed there.  The assumption is more-or-less \cite[(A1-4)]{jasra}.
Recall $\pi(\varphi)$ in \eqref{eq:bpe_def} and we shall write the collection of bounded,  measurable and globally Lipschitz functions $\varphi:\Theta\rightarrow\mathbb{R}$ as $\mathsf{B}_{\text{b,Lip}}(\Theta)$.
By globally Lipschitz we mean that $|\varphi(\theta)-\varphi(\theta')|\leq C|\theta-\theta'|$ for every $(\theta,\theta')\in\Theta^2$ and recall that $|\cdot|$ is the Euclidean norm in the corresponding space. 
The proof of the following result is in Appendix \ref{app:proof}.

\begin{theorem}
\label{thm:2}
Assume (A\ref{assump:coefficients}-\ref{ass:add}).   Then for any 
$(\varphi,T)\in\mathsf{B}_{\text{\emph{b,Lip}}}(\Theta)\times\mathbb{N}$ there exists a $C<+\infty$
such that for any $(\texttt{L},N_0,\dots,N_{\texttt{L}})\in\mathbb{N}^{\texttt{L}+2}$ we have
$$
 \mathbb{E}\left[\left(
\widehat{\mathbb{E}_{\pi^{\texttt{L}}}[\varphi(\theta)]}-\pi(\varphi)
\right)^2\right]  \leq C\left(\frac{1}{N_0+1} +\sum_{\texttt{l}=1}^{\texttt{L}}\frac{ \Delta_{\texttt{l}} }{N_{\texttt{l}}+1}
+ \sum_{\texttt{l}=2}^{\texttt{L}}\sum_{\texttt{q}=1}^{\texttt{l}-1}
\frac{ \Delta_{\texttt{l}}^{1/2}}{N_{\texttt{l}}+1} \frac{ \Delta_{\texttt{q}}^{1/2}}{N_{\texttt{q}}+1} +
\Delta_{\texttt{L}}^2
\right).
$$
\end{theorem}

The main implication of this result is standard in the multilevel Monte Carlo literature.  If one chooses $\epsilon>0$ arbitrarily then to achieve a MSE of $\mathcal{O}(\epsilon^2)$ one can choose $\texttt{L}$ so that $\Delta_{\texttt{L}}=\mathcal{O}(\epsilon)$ and $N_{\texttt{l}}=\mathcal{O}(\epsilon^{-2}\Delta_{\texttt{l}}\texttt{L})$.  The cost to achieve this MSE is
$$
\mathcal{O}\left(\sum_{\texttt{l}=0}^{\texttt{L}}N_{\texttt{l}}\Delta_{\texttt{l}}^{-1}\right) = \mathcal{O}(\epsilon^{-2}\log(\epsilon)^2).
$$
If one had chosen to use only a single level $\texttt{L}$ and the Markov kernel $K_{\texttt{L}}$ the cost to achieve the same MSE is  $\mathcal{O}(\epsilon^{-3})$.

\section{Numerical Results}\label{sec:numerics}

\subsection{Overview}

We evaluate the proposed methodology under a unified sub-diffusive Black--Scholes framework along two complementary data settings. First, we conduct synthetic data experiments generated from the model with known parameter values. This enables us to verify the unbiased estimation provided by the score-based approach and to demonstrate the improvement in computational efficiency achieved by Bayesian-based multilevel MCMC methods.
Second, we consider a real data application based on NVIDIA stock returns, where the model is used as a statistical approximation to observed financial dynamics. The focus is on assessing the ability of the model to reproduce key stylized facts of financial time series, including distributional shape, tail behavior, temporal dependence, and volatility dynamics. 

In our two scenarios, we implement two inference procedures derived from the same methodological framework. The score-based approach is used as the main estimation tool throughout the experiments, while a Bayesian implementation is also included as an alternative numerical realization of the same underlying model-based inference scheme. 
Overall, this experimental design evaluates the proposed framework in both controlled and empirical environments, highlighting its statistical and computational properties as well as its practical applicability in financial modeling.

\subsubsection{Model Specification}

We consider a sub-diffusive Black--Scholes model of time-changed diffusion type, defined by
\begin{equation*}
dX_t = X_t \bigl(\mu\, dL_t + \sigma\, dB_{L_t}\bigr), \qquad X_0 = x > 0,
\end{equation*}
where $L_t$ is the \emph{inverse $\alpha$-stable subordinator} and $B_{L_t}$ denotes a standard Brownian motion evaluated at the random operational time $L_t$.
To make the construction precise, let $D_\alpha=\{D_\alpha(\tau),\tau\ge 0\}$ be an $\alpha$-stable subordinator with index $\alpha\in(0,1)$, i.e. a non-decreasing L\'evy process with Laplace transform
\begin{equation*}
\mathbb{E}\bigl[e^{-\eta D_\alpha(\tau)}\bigr] = e^{-\tau\eta^{\alpha}}, \qquad \eta>0.
\end{equation*}
The inverse $\alpha$-stable subordinator $L_t=\{L_t,t\ge 0\}$ is defined as the first-passage time process of $D_\alpha$:
\begin{equation*}
L_t \;=\; \inf\{\tau>0: D_\alpha(\tau)>t\}, \qquad t\ge 0.
\end{equation*}
The process $L_t$ is non-decreasing and right-continuous, with sample paths exhibiting random periods of constant value that introduce \emph{trapping events} into the dynamics.
As $L_t$ has stationary but dependent increments, $B_{L_t}$ is no longer a L\'evy process; instead, it displays subdiffusive behavior with long-range dependence and fat-tailed marginal distributions.

The observed process is given by
\[
Y_t = X_t + \epsilon_t, \quad \epsilon_t \stackrel{\text{i.i.d.}}{\sim} \mathcal{N}(0,\nu^2),
\]
where $\epsilon_t$ represents independent observational noise which are Gaussian of mean 0 and variance $\nu^2$ (denoted $\mathcal{N}(0,\nu^2)$).
For more background information on this model, see \cite{magdziarz2009black,zhang2025sub}. This stochastic framework is used in both the synthetic data experiments and the empirical study, allowing for a unified evaluation of the proposed estimation procedure.  In this study, $\alpha$ and $\sigma$ are treated as fixed constants 
($\alpha_0$, $\sigma_0$), with the latter being estimable in principle 
but excluded here for parsimony. Estimation focuses on the reduced 
parameter-vector
\(\theta = (\mu, \nu^2).\)

\subsection{Synthetic Data Experiments}

We now introduce a set of synthetic data experiments to evaluate estimation procedure by setting
$\alpha = 0.75$, $\mu = 0.10$, $\sigma = 1$ and $\nu^2=0.10$ and simulating 100 data points with discretization level $\texttt{L}=8$.

\subsubsection{Score-Based Estimation}

We will estimate $\theta$ using $M\in\{2,4,\dots,32\}$ samples,  that is running Algorithm \ref{alg:umsa} $M$ times in parallel and averaging the results.  This is repeated 100 times to compute an approximation of the MSE.
We use $N=100$ samples in the conditional particle filter and conditional delta particle filter.
We truncate the support of $\texttt{l}$ to $\{3,\dots,8\}$ and $\mathbb{P}_{\texttt{L}}(\texttt{l})\propto 2^{-1.5\texttt{l}}$
for $\texttt{l}\in\{3,\dots,8\}$. $N_{\texttt{p}}=5\times2^{\texttt{p}}$ and we allow $\mathbb{P}_{\texttt{p}}$ to depend upon $\texttt{l}$ and in particular
\[
\mathbb{P}_{\texttt{P}}(\texttt{p}|\texttt{l}) \propto g(\texttt{p}|\texttt{l}),
\]
with
\[
g(\texttt{p}|\texttt{l}) =
\begin{cases}
2^{5-\texttt{p}}, & \texttt{p} \in \{1, \dots, 5 \wedge (8 - \texttt{l})\}, \\
2^{-\texttt{p}} \texttt{p} [\log_2(\texttt{p})]^2, & 5 < \texttt{p} \leq 7 \quad \text{and} \quad \texttt{l}=3.\\
0,& \text{otherwise}.
\end{cases}
\]
%
%
%
%

In Figure \ref{fig:unbiased}, we report the relationship between the MSE, CPU runtime, and the number of parallel runs $M$. The runtime is computed as the total CPU time. The results indicate that the MSE scales approximately as 
$\mathcal{O}(M^{-1})$, which is consistent with the theoretical property the estimator is unbiased with variance
$\mathcal{O}(M^{-1})$. 

\begin{figure}[h]
\centering
\includegraphics[width=0.48\textwidth]{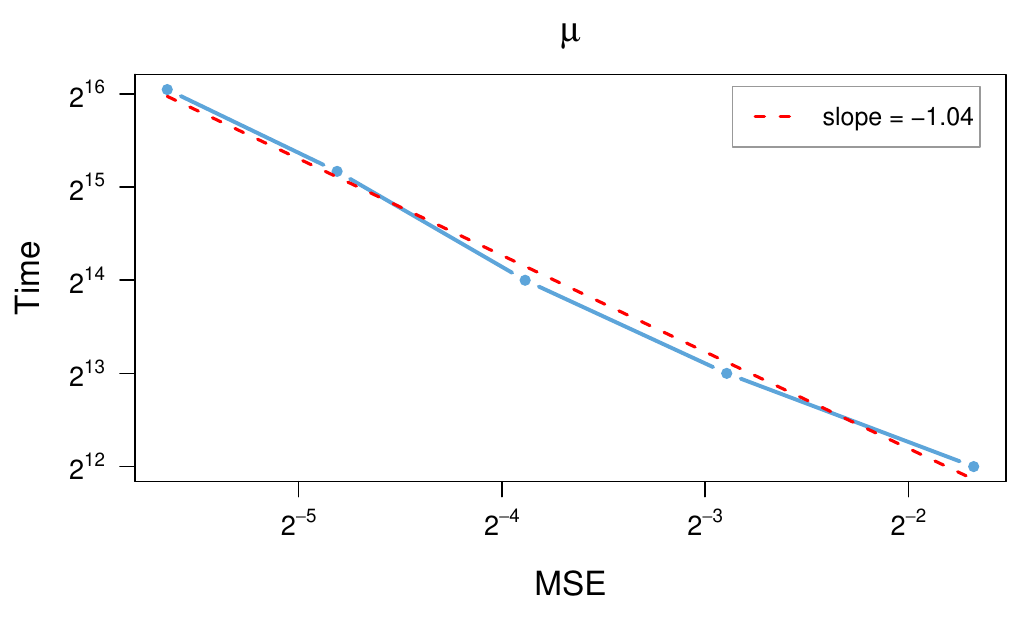}
\includegraphics[width=0.48\textwidth]{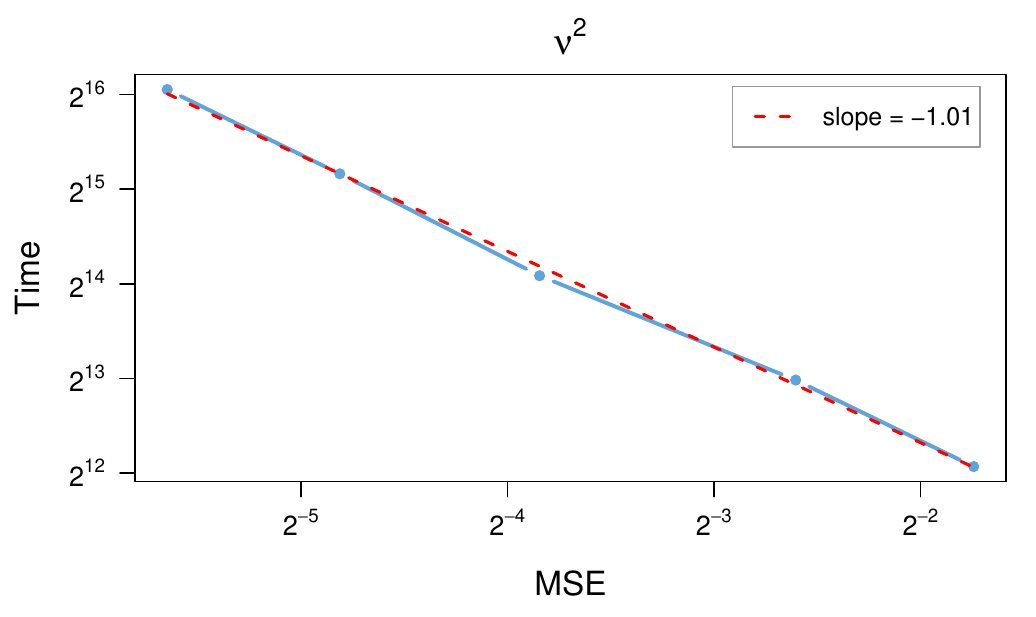}\\[6pt]
\includegraphics[width=0.48\textwidth]{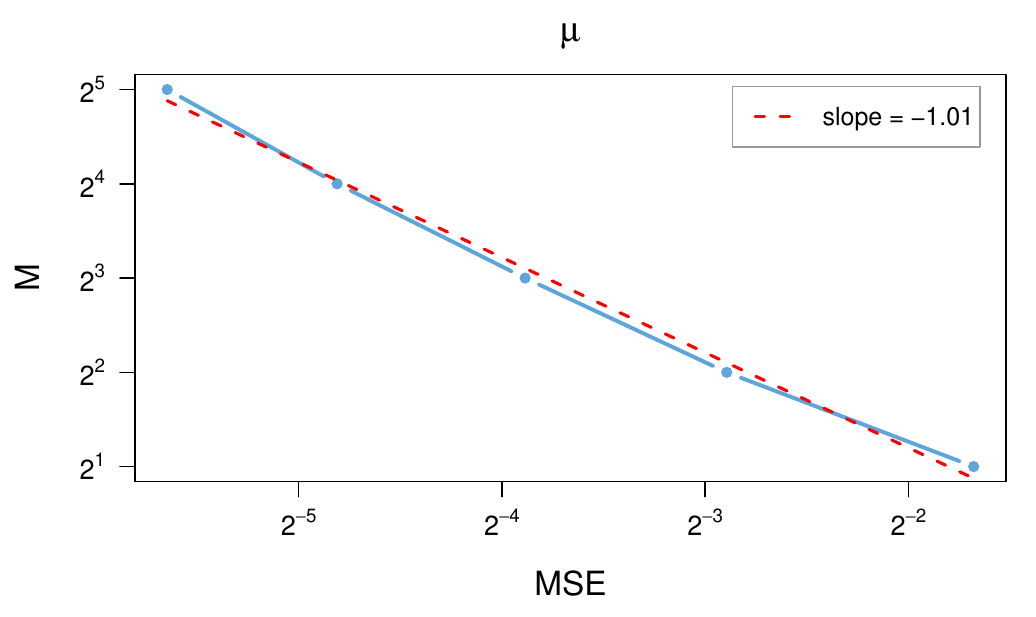}
\includegraphics[width=0.48\textwidth]{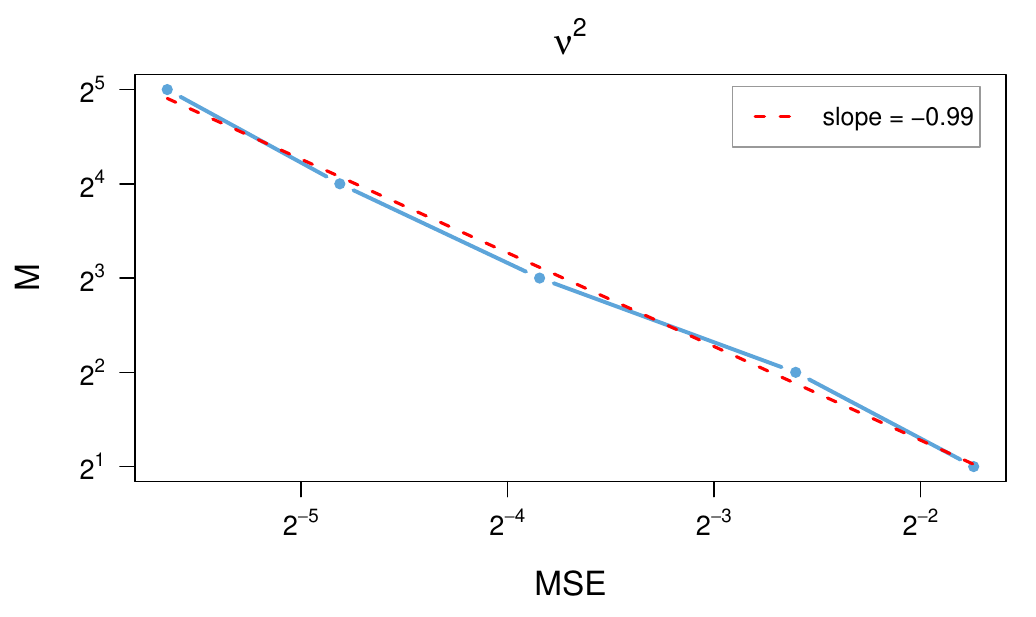}
\caption{Top row: MSE versus CPU runtime for estimating $\mu$ (left) and $\nu^2$ (right). Bottom row: MSE versus $M$ for $\mu$ (left) and $\nu^2$ (right). Dashed lines indicate fitted slopes in log-log scale.}
\label{fig:unbiased}
\end{figure}

\subsubsection{Bayesian Estimation}
We will use the particle marginal Metropolis-Hastings method  to be the type of \(K_1\) and \(\check{K}_1\) in Algorithm \ref{alg:bayes}. More detailed setting includes independent Gaussian priors for \(\mu,\nu^2\) and Gaussian random walk proposals with variance $0.1$, with a burn-in proportion of 0.2  (i.e., the first 20\% of iterations are discarded). The result is in Figure \ref{fig:ml_sl_complexity}. Based on our methodology, the multilevel method significantly reduces computational cost. The regression coefficients between log cost and log MSE empirically validate our theoretical results.

\begin{figure}[H]
\centering
\includegraphics[width=0.48\textwidth]{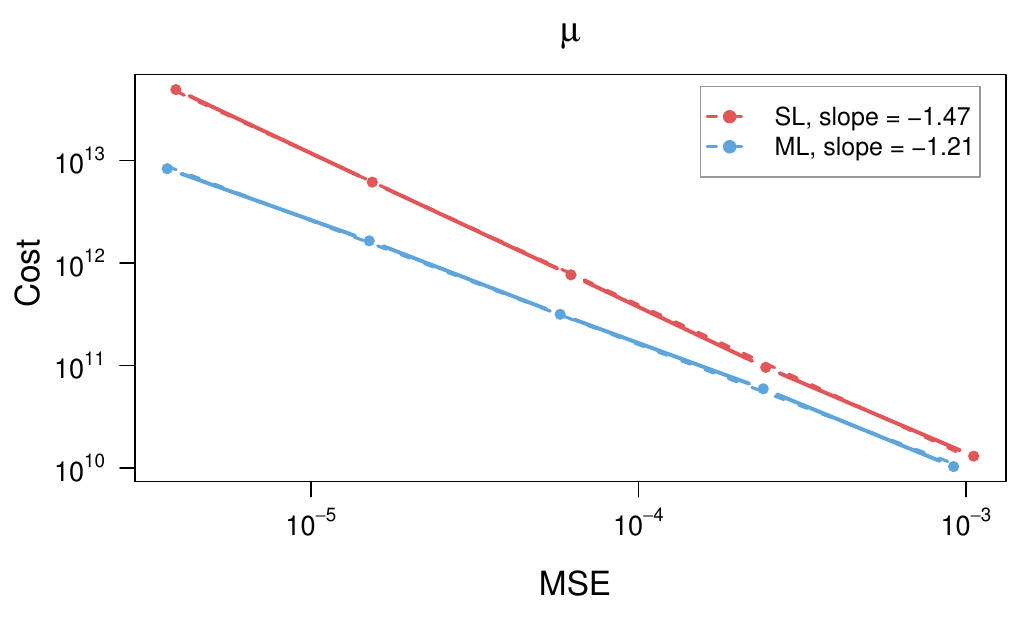}
\includegraphics[width=0.48\textwidth]{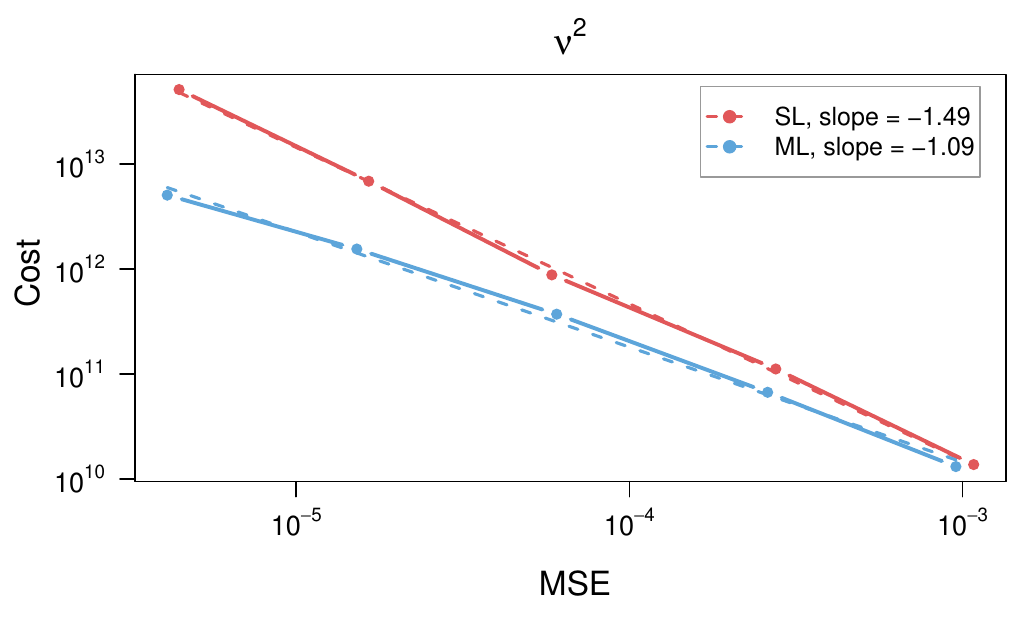}
\caption{Complexity comparison of single-level (SL) and multilevel (ML) PMMH: MSE versus computational cost for $\mu$ (left) and $\nu^2$ (right).}
\label{fig:ml_sl_complexity}
\end{figure}


\subsection{Real Data}

In this section, we evaluate the proposed sub-diffusion-driven Black-Scholes model using NVIDIA stock data. The empirical dataset consists of daily observations of NVIDIA stock prices from November 1, 2021 to October 31, 2022. After excluding non-trading days such as weekends and public holidays, the sample contains approximately 252 trading observations (see Figure \ref{fig:price}). This period is relevant as it covers a pronounced bearish market phase for NVIDIA, during which the stock experienced a substantial decline following the post-pandemic technology sector correction and broader macroeconomic tightening. Data was downloaded from Kaggle: \url{https://www.kaggle.com/datasets/alitaqishah/nvidia-stock-data-19992026-the-ai-mega-stock}.

Our empirical analysis proceeds in two stages. First, we estimate the model parameters using two methods: the score-based estimation and Bayesian estimation. Second, using the estimated parameters from both methods, we simulate trajectories from the sub-diffusion-driven Black-Scholes model and compare the resulting synthetic returns with the empirical NVIDIA returns. We examine both distributional characteristics and temporal dependence structures, with particular attention to whether the model can reproduce key stylized facts of financial returns such as heavy tails, volatility clustering, and persistent autocorrelation under adverse market conditions characterized by elevated volatility and structural stress.
We carefully chose
$\alpha = 0.75$ and $\sigma = 0.10$. We set $\texttt{L}=6$ in the Bayesian method. The estimated parameters are in Table \ref{tab:parameter_estimates}

\begin{figure}[h]
    \centering
\includegraphics[width=0.8\textwidth,height=10cm]{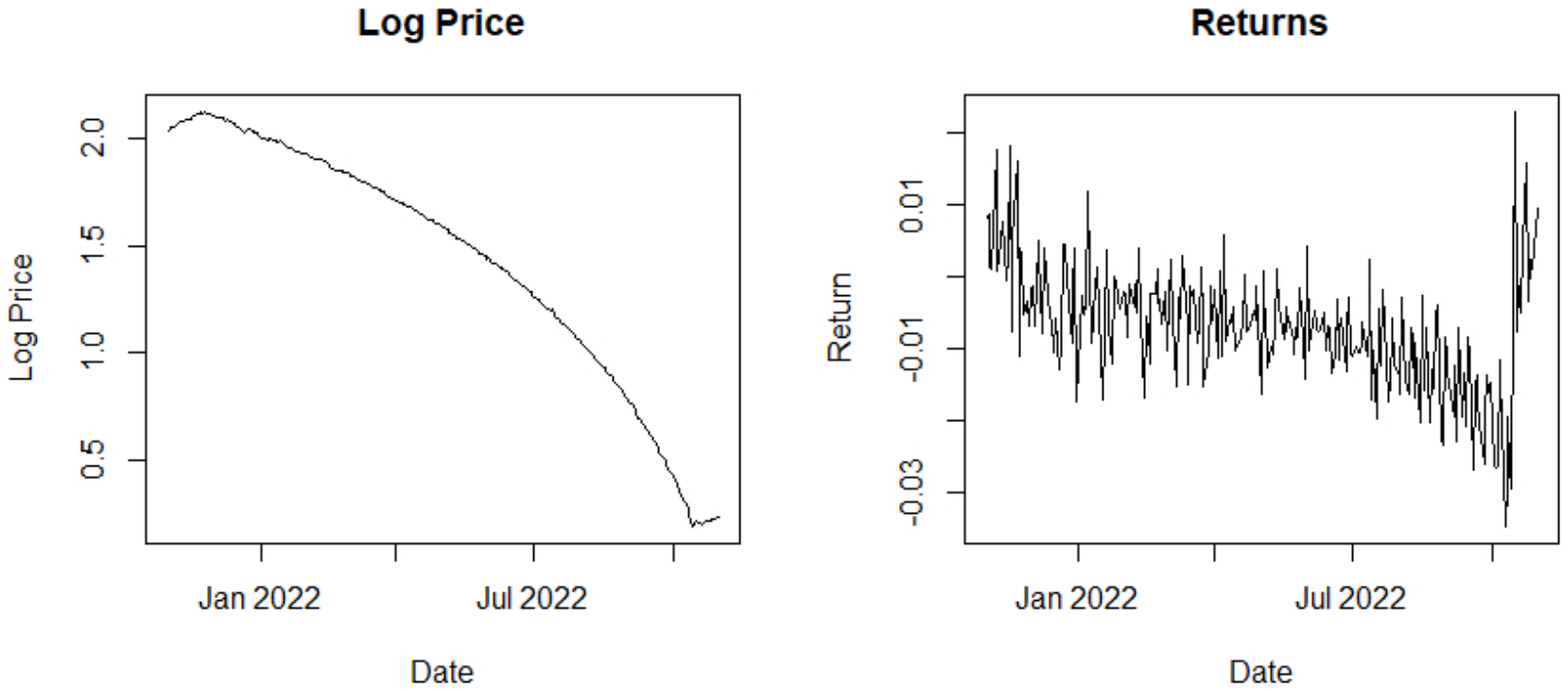}
    \caption{NVIDIA log price and returns (\(R_t=X_t-X_{t-1}\) where \(X_t\) is the log-price at time $t$).}
    \label{fig:price}
\end{figure}

\begin{table}[htbp]
\centering
\begin{tabular}{lcc}
\hline
Method & $\mu$ & $\nu^2$ \\
\hline
Score-based & -0.04 & 1.00 \\
Bayesian  & -0.038 & 0.90 \\
\hline
\end{tabular}
\caption{Parameter estimates for $\mu$ and $\nu^2$ using score-based and Bayesian methods.}\label{tab:parameter_estimates}
\end{table}

We summarize the main statistical properties of the empirical and simulated return series in Table \ref{tb:stats} and Table \ref{tb:quantiles}. The estimates obtained from the two methods do not appear to differ substantially; in this experiment, our focus is primarily on the intrinsic properties of the model.

Regarding the simulation of returns, the simulated mean and standard deviation track their empirical counterparts quite closely (see Table \ref{tb:stats}). However, the simulated returns exhibit a more pronounced negative skewness than the real data, which is nearly symmetric. This suggests that the subdiffusion mechanism may overstate downside risk by generating a heavier left tail, possibly as a consequence of the memory effects or non-Markovian waiting-time dynamics embedded in the model. This distortion is most clearly visible in the quantile comparison reported in Table \ref{tb:quantiles}: while the lower quantiles (1\% and 5\%) remain reasonably close to the empirical values, the upper tail (95\% and 99\%) is markedly compressed, with the simulated quantiles falling well below their real-data counterparts. Additionally, kurtosis estimates are somewhat inflated relative to the empirical value. This excess kurtosis indicates that the simulated distribution is more leptokurtic than the actual returns, implying that the model overcompensates for tail risk and produces more extreme observations than observed in the bearish regime.

 We consider the autocorrelation function (ACF); in the context of financial returns, the ACF of absolute (or squared) returns is of particular interest, as it captures the persistence of volatility shocks---commonly referred to as volatility clustering.
Figure~\ref{fig:ACF} reports the sample ACF of absolute returns for the empirical NVIDIA data and the two simulated series. Across lags $1$ through $20$, the sub-diffusion-driven model maintains ACFs that remain well above the $95\%$ confidence bounds (dashed red lines), closely tracking the persistent decay pattern observed in the real data. This indicates that the model successfully reproduces the long-memory property of volatility: large price movements tend to be followed by further large movements, and the influence of past shocks decays only gradually rather than instantaneously. Such behavior is a direct consequence of the non-Markovian, heavy-tailed waiting-time dynamics embedded in the sub-diffusion framework, which naturally generates temporal dependence in second moments without requiring additional ad-hoc specifications.

 Overall, our experiments demonstrate that this model is capable of capturing the characteristics of bear markets, as stated in \cite{zhang2025sub}, albeit with certain limitations. On the other hand, we also acknowledge the constraints of our experimental design; for instance, we did not conduct inference for all parameters, particularly $\alpha$, which may be addressed in future work.
\begin{table}[h]
\centering
\begin{minipage}{0.48\textwidth}
\centering
\small
\begin{tabular}{lccc}
\hline
Statistic & Real & Score-based & Bayesian \\
\hline
Mean     & -0.00695 & -0.00694 & -0.00864 \\
Std Dev  &  0.00831 &  0.00700 &  0.00894 \\
Skewness & -0.00709 & -1.38129 & -1.50982 \\
Kurtosis &  4.38377 &  5.67742 &  7.63829 \\
\hline
\end{tabular}
\caption{Summary statistics}\label{tb:stats}
\end{minipage}
\hfill
\begin{minipage}{0.48\textwidth}
\centering
\small
\begin{tabular}{lccc}
\hline
Quantile & Real & Score-based & Bayesian \\
\hline
1\%  & -0.02774 & -0.02800 & -0.03714 \\
5\%  & -0.02054 & -0.02035 & -0.02357 \\
50\% & -0.00683 & -0.00503 & -0.00661 \\
95\% &  0.00512 &  0.00101 &  0.00153 \\
99\% &  0.01669 &  0.00372 &  0.00480 \\
\hline
\end{tabular}
\caption{Quantile comparison}\label{tb:quantiles}
\end{minipage}
\label{tab:comparison}
\end{table}

\begin{figure}[h]
\centering
\includegraphics[width=0.9\textwidth,height=12cm]{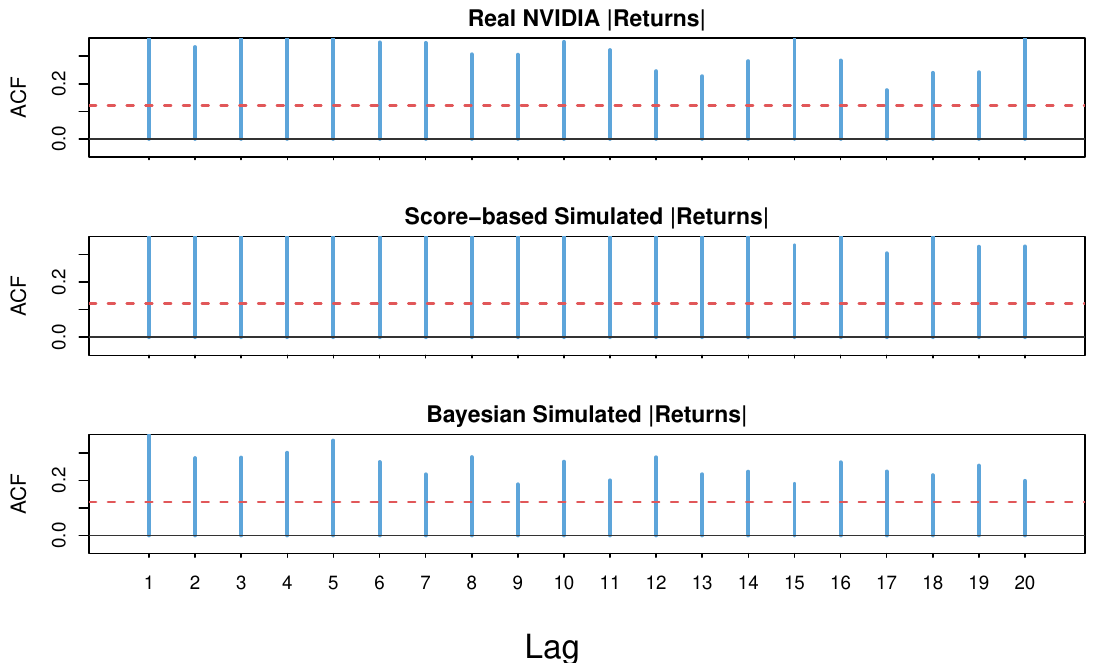}
\caption{Autocorrelation functions of absolute returns Top: real NVIDIA data; Middle: score-based simulation; Bottom: Bayesian simulation. Dashed red lines denote 95\% confidence bounds.}\label{fig:ACF}
\end{figure}




\appendix

\section{Proof of Main Result}\label{app:proof}

\subsection{Additional Assumption}

In order to give our additional assumption,  we note that by \cite[Theorem 3.1]{jum2014strong} one has that
for any $k\in\{1,\dots,T\}$ there exists a $C<+\infty$ such that
for any $(\theta,\varphi,\texttt{l},\texttt{q})\in\Theta\times\mathsf{B}_{\text{b,Lip}}(\Theta\times\mathbb{R}^{d_x})\times\mathbb{N}\times\{1,2\}$
\begin{equation}\label{eq:strong_weak}
\check{\mathbb{E}}_{\theta}\left[|\varphi(\theta,X_k^{\texttt{l}})-\varphi(\theta,X_k^{\texttt{l-1}})|^{\texttt{q}}\right]^{3-\texttt{q}} \leq C\max\{\|\varphi\|_{\text{Lip}},\|\varphi\|_{\text{Lip}}^2\}\Delta_{\texttt{l}}
\end{equation}
where 
$\|\varphi\|_{\text{Lip}}$ is the Lipschitz constant for $\varphi$ and
$\check{\mathbb{E}}_{\theta}$ is an expectation with respect to the law of the process coupled with the synchronous coupling at any level $\texttt{l}\in\mathbb{N}$ with $\theta$ and $x_0$ fixed.  We remark that the constant $C$ in the above could depend on $x_0$ but not on $\theta$.   We remark that the result of  \cite[Theorem 3.1]{jum2014strong} are easily extended to the autonoumous coefficient case of this paper.  Below $\text{Lip}(\mathsf{Z})$ are the collection of real-valued globally lipschitz functions on some space $\mathsf{Z}$.
We will suppose that the Markov kernels $K_{0}$ and $\check{K}_{\texttt{l}}$ operate on some measurable space
$(\mathsf{K}_0,\mathcal{K})$ and $(\mathsf{K}_{\texttt{l}},\mathcal{K}_{\texttt{l}})$ with invariant measures
$\Upsilon_0$ and $\Upsilon_{\texttt{l}}$.  
The abstract spaces and invariant measures are to ensure that one can use
advanced MCMC kernels on extended state-spaces,  but we note that the invariant measures must be related (or equal) to  the $\pi^0$ and $\pi^{\texttt{l}}$.

\begin{hypA}\label{ass:add}
\begin{enumerate}
\item{For $y\in\mathsf{Y}$,  there exists $0<\underline{C}<\overline{C}<+\infty$ such that for any $(\theta,x)\in\Theta\times\mathbb{R}^{d_x}$
$$
\underline{C}\leq G_{\theta}(x,y) \leq \overline{C}.
$$
For each $y\in\mathsf{Y}$,  $G_{\theta}(x,y)\in\text{Lip}(\Theta\times\mathbb{R}^{d_x})$.
}
\item{The constant $C(x_0)$ in \eqref{eq:strong_weak} is such that $\int_{\Theta}\int_{\mathbb{R}^{d_x}}C(x_0)
\mu((x_0)dx_0\overline{\pi}(\theta)d\theta<+\infty$.
}
\item{There exist a constant $C\in(0,1)$ and probability measures $\zeta_0,\zeta_{\texttt{l}}$ such that for
any $\tilde{u}_{0}\in\mathsf{K}_0$ and $\tilde{u}_{\texttt{l}}\in\mathsf{K}_{\texttt{l}}$, $\texttt{l}\in\mathbb{N}$
$$
K_0(\tilde{u}_{0},d\tilde{u}_0') \geq C\zeta_0(d\tilde{u}_0') \quad\text{and}\quad \check{K}_{\texttt{l}}(\tilde{u}_{\texttt{l}},d\tilde{u}_{\texttt{l}}') \geq C\zeta_{\texttt{l}}(d\tilde{u}_{\texttt{l}}').
$$
In addition $K_0$ and $\check{K}_{\texttt{l}}$ are $\Upsilon_0$ and $\Upsilon_{\texttt{l}}$ respectively,  $\texttt{l}\in\mathbb{N}$.
}
\item{For $\varphi\in\mathsf{B}_{\text{b,Lip}}(\Theta)$ there exists a $C<+\infty$ such that for any $\texttt{l}\in\mathbb{N}_0$
$$
|\mathbb{E}_{\pi^{\texttt{l}}}[\varphi(\theta)] - \mathbb{E}_{\pi}[\varphi(\theta)]|\leq C\Delta_{\texttt{l}}.
$$
}
\end{enumerate}
\end{hypA}

Assumptions (A\ref{ass:add}) 1.~\& 3.~are \cite[(A1) \& (A3)]{jasra}.  (A\ref{ass:add}) 2.~is related to \cite[(A2)]{jasra} except that the forward type rates are already established in \cite[Theorems 3.2, 3.3]{jum2014strong} and we simply need to ensure integrability of the constant,  which one expects relates to having at most polynomial moments of $\mu_{\theta}$.  (A\ref{ass:add}) 4.~is \cite[(A4)]{jasra} and in order to verify it one could start with adapting \cite[Theorem 3.3]{jum2014strong},  but one must be careful in trying to embed a marginal weak error result
into the context of our model,  so just as in \cite{jasra} we make an assumption;  we remark that by \eqref{eq:strong_weak} the result for an upper-bound of $\mathcal{O}(\Delta_{\texttt{l}}^{1/2})$ would be straight-forward to obtain,  but we believe this is not sharp.  
Throughout our proofs,  as was the case in \cite{jasra},  it is assumed all the initial distributions of the Markov chains are the invariant measures.  

\subsection{Proof of Theorem \ref{thm:2}}

To prove our main theorem we need the following technical lemma which is basically \cite[Lemma A.2.]{jasra}.
\begin{lem}\label{lem:lem1}
Assume (A\ref{assump:coefficients}),(A\ref{ass:add} 1.),  (A\ref{ass:add} 2.).  Then for
any $\varphi\in\mathsf{B}_{\text{\emph{b,Lip}}})(\Theta)$ there exists a $C<+\infty$ such
that for any $(\texttt{l},q)\in\mathbb{N}\times\{1,2\}$
$$
\mathbb{E}_{\pi^{\texttt{l}}}[|\varphi(\theta)\mathcal{G}_{\theta,\text{f}}(\mathbf{U}_{0:T}^{\texttt{l}})]-\varphi(\theta)\mathcal{G}_{\theta,\text{c}}(\mathbf{U}_{0:T}^{\texttt{l}})]|^{\texttt{q}}]^{3-\texttt{q}}\leq C\Delta_{\texttt{l}}.
$$
\end{lem}

\begin{proof}
Using \eqref{eq:strong_weak} and (A\ref{ass:add} 2.) one can deduce a version of \cite[(A2)]{jasra} in the context of this paper.  From there the proof follows \cite[Lemma A.2.]{jasra} and is thus omitted.
\end{proof}

\begin{proof}[Proof of Theorem \ref{thm:2}]
The following proof follows similar ones in \cite{ml_cont,ml_mv_sde} adapted to the scenario of this article.
We start by noting that
\begin{equation}\label{eq:master}
\widehat{\mathbb{E}_{\pi^{\texttt{L}}}[\varphi(\theta)]}-\pi(\varphi) = \sum_{j=1}^3 \mathtt{T}_j
\end{equation}
where
\begin{eqnarray*}
\mathtt{T}_1 & = & \frac{1}{N_0+1}\sum_{n=0}^{N_0}\varphi(\theta_n^0) - \mathbb{E}_{\pi^0}[\varphi(\theta)]\\
\mathtt{T}_2 & = & \sum_{\texttt{l}=1}^{\texttt{L}}
\Bigg\{
\frac{\sum_{n=0}^{N_{\texttt{l}}}
\varphi(\theta_n^{\texttt{l}})\mathcal{G}_{\theta_n^{\texttt{l}},\text{f}}(\mathbf{u}^{n,\texttt{l}}_{0:T})
}{\sum_{n=0}^{N_{\texttt{l}}}\mathcal{G}_{\theta_n^{\texttt{l}},\text{f}}(\mathbf{u}^{n,\texttt{l}}_{0:T})} - 
\frac{\sum_{n=0}^{N_{\texttt{l}}}
\varphi(\theta_n^{\texttt{l}})\mathcal{G}_{\theta_n^{\texttt{l}},\text{c}}(\mathbf{u}^{n,\texttt{l}}_{0:T})
}{\sum_{n=0}^{N_{\texttt{l}}}\mathcal{G}_{\theta_n^{\texttt{l}},\text{c}}(\mathbf{u}^{n,\texttt{l}}_{0:T})}
-\left\{
 \mathbb{E}_{\pi^{\texttt{l}}}[\varphi(\theta)]-
\mathbb{E}_{\pi^{\texttt{l}-1}}[\varphi(\theta)]
\right\}
\Bigg\}\\
\mathtt{T}_3 & = &  \mathbb{E}_{\pi^{\texttt{L}}}[\varphi(\theta)]-  \mathbb{E}_{\pi}[\varphi(\theta)].
\end{eqnarray*}
By applying the $C_2-$inequality twice we can deal with the second moment of each of the three terms separately.

In the case of $\mathtt{T}_1$ standard results on reversible,  uniformly ergodic Markov chains yields that
\begin{equation}\label{eq:t1_bound}
\mathbb{E}[\mathtt{T}_1^2] \leq \frac{C}{N_0+1}
\end{equation}
where $C$ does not depend on $N_0$ but it does depend on $\varphi$.  From herein $C$ will denote a constant whose value will change from line-to-line but does not depend on $\texttt{l}\in\mathbb{N}$.

For  $\mathtt{T}_2$ we start by noting that
$$
 \mathbb{E}_{\pi^{\texttt{l}}}[\varphi(\theta)]-
\mathbb{E}_{\pi^{\texttt{l}-1}}[\varphi(\theta)]
= \frac{\mathbb{E}_{\check{\pi}^{\texttt{l}}}[\varphi(\theta)\mathcal{G}_{\theta,\text{f}}(\mathbf{U}_{0:T}^{\texttt{l}})]}
{\mathbb{E}_{\check{\pi}^{\texttt{l}}}[\mathcal{G}_{\theta,\text{f}}(\mathbf{U}_{0:T}^{\texttt{l}})]}  -
\frac{\mathbb{E}_{\check{\pi}^{\texttt{l}}}[\varphi(\theta)\mathcal{G}_{\theta,\text{c}}(\mathbf{U}_{0:T}^{\texttt{l}})]}
{\mathbb{E}_{\check{\pi}^{\texttt{l}}}[\mathcal{G}_{\theta,\text{c}}(\mathbf{U}_{0:T}^{\texttt{l}})]}.
$$
so we will write
$$
\mathtt{T}_2 = \sum_{\texttt{l}=1}^{\texttt{L}} \mathtt{D}_{\texttt{l}}
$$
where
$$
\mathtt{D}_{\texttt{l}} = \frac{\sum_{n=0}^{N_{\texttt{l}}}
\varphi(\theta_n^{\texttt{l}})\mathcal{G}_{\theta_n^{\texttt{l}},\text{f}}(\mathbf{u}^{n,\texttt{l}}_{0:T})
}{\sum_{n=0}^{N_{\texttt{l}}}\mathcal{G}_{\theta_n^{\texttt{l}},\text{f}}(\mathbf{u}^{n,\texttt{l}}_{0:T})} - 
\frac{\sum_{n=0}^{N_{\texttt{l}}}
\varphi(\theta_n^{\texttt{l}})\mathcal{G}_{\theta_n^{\texttt{l}},\text{c}}(\mathbf{u}^{n,\texttt{l}}_{0:T})
}{\sum_{n=0}^{N_{\texttt{l}}}\mathcal{G}_{\theta_n^{\texttt{l}},\text{c}}(\mathbf{u}^{n,\texttt{l}}_{0:T})}
- \frac{\mathbb{E}_{\check{\pi}^{\texttt{l}}}[\varphi(\theta)\mathcal{G}_{\theta,\text{f}}(\mathbf{U}_{0:T}^{\texttt{l}})]}
{\mathbb{E}_{\check{\pi}^{\texttt{l}}}[\mathcal{G}_{\theta,\text{f}}(\mathbf{U}_{0:T}^{\texttt{l}})]}  +
\frac{\mathbb{E}_{\check{\pi}^{\texttt{l}}}[\varphi(\theta)\mathcal{G}_{\theta,\text{c}}(\mathbf{U}_{0:T}^{\texttt{l}})]}
{\mathbb{E}_{\check{\pi}^{\texttt{l}}}[\mathcal{G}_{\theta,\text{c}}(\mathbf{U}_{0:T}^{\texttt{l}})]}.
$$
Then we have
\begin{equation}\label{eq:t2_master}
\mathbb{E}[\mathtt{T}_2^2] = \sum_{\texttt{l}=1}^{\texttt{L}} \mathbb{E}[\mathtt{D}_{\texttt{l}}^2] +
2\sum_{\texttt{l}=2}^{\texttt{L}}\sum_{\texttt{q}=1}^{\texttt{l}}\mathbb{E}[\mathtt{D}_{\texttt{l}}]\mathbb{E}[\mathtt{D}_{\texttt{q}}].
\end{equation}
To continue with our argument,  we recall the simple identity for real numbers $a,A,\dots,d,D$ with $A,B,C,D$ all
non-zero:
\begin{equation}
\begin{split}
    \frac{a}{A} - \frac{b}{B} - \frac{c}{C} + \frac{d}{D} =\,&\frac{1}{A}(a-b-c+d) -\frac{b}{AB}(A-B-C+D) - \frac{1}{AC}(A-C)(c-d)\\
    -\,&\frac{1}{AB}(C-D)(b-d)+\frac{d}{CBD}(B-D)(C-D) + \frac{d}{ACB}(A-C)(C-D). \\
\end{split}
\label{eq:abcd}
\end{equation}
\eqref{eq:abcd} can be used to bound $\mathbb{E}[\mathtt{D}_{\texttt{l}}^2]$ and $\mathbb{E}[\mathtt{D}_{\texttt{l}}]$ by using the $C_2-$inequality or the triangular inequality five times.  As most of the bounds of these terms are similar we only consider terms of the type $(a-b-c+d)/A$ and $(A-C)(c-d)/(AC)$ as 
bounding the other terms have similar proofs.  We begin with $\mathbb{E}[\mathtt{D}_{\texttt{l}}^2]$  and the corresponding $(a-b-c+d)/A$ term which is
$$
\mathbb{E}\left[
\left(
\frac{
\frac{1}{N_{\texttt{l}}+1}
\sum_{n=0}^{N_{\texttt{l}}}\left\{
\varphi(\theta_n^{\texttt{l}})\mathcal{G}_{\theta_n^{\texttt{l}},\text{f}}(\mathbf{u}^{n,\texttt{l}}_{0:T})
-
\varphi(\theta_n^{\texttt{l}})\mathcal{G}_{\theta_n^{\texttt{l}},\text{c}}(\mathbf{u}^{n,\texttt{l}}_{0:T})\right\}-
\left\{
\mathbb{E}_{\check{\pi}^{\texttt{l}}}[\varphi(\theta)\mathcal{G}_{\theta,\text{f}}(\mathbf{U}_{0:T}^{\texttt{l}})-
\varphi(\theta)\mathcal{G}_{\theta,\text{c}}(\mathbf{U}_{0:T}^{\texttt{l}})]
\right\}}
{\tfrac{1}{N_{\texttt{l}}+1}\sum_{n=0}^{N_{\texttt{l}}}\mathcal{G}_{\theta_n^{\texttt{l}},\text{f}}(\mathbf{u}^{n,\texttt{l}}_{0:T})}
\right)^2
\right].
$$
By (A\ref{ass:add} 1.) the denominator is lower-bounded by a deterministic constant and the resulting expression can be controlled by using \cite[Proposition A.1.]{jasra} along with Lemma \ref{lem:lem1} to deduce the upper-bound:
$$
\frac{C\Delta_{\texttt{l}}}{N_{\texttt{l}}+1}.
$$
For the $(A-C)(c-d)/(AC)$ term we note that again the denominator is lower-bounded by a deterministic constant 
so one has to control a term of the type
$$
C\mathbb{E}_{\check{\pi}^{\texttt{l}}}[\varphi(\theta)\mathcal{G}_{\theta,\text{f}}(\mathbf{U}_{0:T}^{\texttt{l}})-
\varphi(\theta)\mathcal{G}_{\theta,\text{c}}(\mathbf{U}_{0:T}^{\texttt{l}})]^2
\mathbb{E}\left[\left(
\frac{1}{N_{\texttt{l}}+1}\sum_{n=0}^{N_{\texttt{l}}}\mathcal{G}_{\theta_n^{\texttt{l}},\text{f}}(\mathbf{u}^{n,\texttt{l}}_{0:T}) - \mathbb{E}_{\check{\pi}^{\texttt{l}}}[\varphi(\theta)\mathcal{G}_{\theta,\text{f}}(\mathbf{U}_{0:T}^{\texttt{l}})]
\right)^2\right].
$$
We can apply Lemma \ref{lem:lem1} to the left expression and standard results for uniformly ergodic Markov chains giving an upper-bound of the type  $C\Delta_{\texttt{l}}/(N_{\texttt{l}}+1)$.  Thus our arguments are enough to obtain that
\begin{equation}\label{eq:dl_bound}
\mathbb{E}[\mathtt{D}_{\texttt{l}}^2] \leq \frac{C\Delta_{\texttt{l}}}{N_{\texttt{l}}+1}.
\end{equation}
For the case of $\mathbb{E}[\mathtt{D}_{\texttt{l}}]$ and the $(a-b-c+d)/A$ term we can write this as
$$
\mathbb{E}\Bigg[
\left\{\frac{1}{\tfrac{1}{N_{\texttt{l}}+1}\sum_{n=0}^{N_{\texttt{l}}}\mathcal{G}_{\theta_n^{\texttt{l}},\text{f}}(\mathbf{u}^{n,\texttt{l}}_{0:T})}-
\frac{1}{\mathbb{E}_{\check{\pi}^{\texttt{l}}}[\mathcal{G}_{\theta,\text{f}}(\mathbf{U}_{0:T}^{\texttt{l}})]}
\right\}\times
$$
$$
\Bigg\{
\frac{1}{N_{\texttt{l}}+1}
\sum_{n=0}^{N_{\texttt{l}}}\left\{
\varphi(\theta_n^{\texttt{l}})\mathcal{G}_{\theta_n^{\texttt{l}},\text{f}}(\mathbf{u}^{n,\texttt{l}}_{0:T})
-
\varphi(\theta_n^{\texttt{l}})\mathcal{G}_{\theta_n^{\texttt{l}},\text{c}}(\mathbf{u}^{n,\texttt{l}}_{0:T})\right\}-
\left\{
\mathbb{E}_{\check{\pi}^{\texttt{l}}}[\varphi(\theta)\mathcal{G}_{\theta,\text{f}}(\mathbf{U}_{0:T}^{\texttt{l}})-
\varphi(\theta)\mathcal{G}_{\theta,\text{c}}(\mathbf{U}_{0:T}^{\texttt{l}})]
\right\}\Bigg\}
\Bigg]
$$
as the Markov chain is started in stationarity.  Then one can use Cauchy-Schwarz and standard results for uniformly ergodic Markov chains for the left term and \cite[Proposition A.1.]{jasra} along with Lemma \ref{lem:lem1}  to the right term to give an upper-bound of the type $C\Delta_{\texttt{l}}^{1/2}/(N_{\texttt{l}}+1)$.  Then for the  $(A-C)(c-d)/(AC)$ using a similar argument one has that in expectation it is equal to
$$
\mathbb{E}_{\check{\pi}^{\texttt{l}}}[\varphi(\theta)\mathcal{G}_{\theta,\text{f}}(\mathbf{U}_{0:T}^{\texttt{l}})-
\varphi(\theta)\mathcal{G}_{\theta,\text{c}}(\mathbf{U}_{0:T}^{\texttt{l}})]\frac{1}{\mathbb{E}_{\check{\pi}^{\texttt{l}}}[\mathcal{G}_{\theta,\text{f}}(\mathbf{U}_{0:T}^{\texttt{l}})]}\times
$$
$$
\mathbb{E}\Bigg[
\left\{\frac{1}{\tfrac{1}{N_{\texttt{l}}+1}\sum_{n=0}^{N_{\texttt{l}}}\mathcal{G}_{\theta_n^{\texttt{l}},\text{f}}(\mathbf{u}^{n,\texttt{l}}_{0:T})}-
\frac{1}{\mathbb{E}_{\check{\pi}^{\texttt{l}}}[\mathcal{G}_{\theta,\text{f}}(\mathbf{U}_{0:T}^{\texttt{l}})]}
\right\}\left\{
\frac{1}{N_{\texttt{l}}+1}\sum_{n=0}^{N_{\texttt{l}}}\mathcal{G}_{\theta_n^{\texttt{l}},\text{f}}(\mathbf{u}^{n,\texttt{l}}_{0:T}) - \mathbb{E}_{\check{\pi}^{\texttt{l}}}[\mathcal{G}_{\theta,\text{f}}(\mathbf{U}_{0:T}^{\texttt{l}})]
\right\}
\Bigg].
$$
The top line can be dealt with by using Lemma \ref{lem:lem1}  and (A\ref{ass:add} 1.)  and the second line by 
Cauchy-Schwarz and standard results for uniformly ergodic Markov chains giving an upper-bound of the type
$C\Delta_{\texttt{l}}^{1/2}/(N_{\texttt{l}}+1)$.  As a result  we have 
\begin{equation}\label{eq:dl_bias_bound}
\mathbb{E}[\mathtt{D}_{\texttt{l}}] \leq \frac{C\Delta_{\texttt{l}}^{1/2}}{N_{\texttt{l}}+1}.
\end{equation}
Combining \eqref{eq:dl_bound} and \eqref{eq:dl_bias_bound} with \eqref{eq:t2_master} yields that 
\begin{equation}\label{eq:t2_bound}
\mathbb{E}[\texttt{T}_2^2] \leq C
\left(\sum_{\texttt{l}=1}^{\texttt{L}}\frac{ \Delta_{\texttt{l}} }{N_{\texttt{l}}+1}
+ \sum_{\texttt{l}=2}^{\texttt{L}}\sum_{\texttt{q}=1}^{\texttt{l}-1}
\frac{ \Delta_{\texttt{l}}^{1/2}}{N_{\texttt{l}}+1} \frac{ \Delta_{\texttt{q}}^{1/2}}{N_{\texttt{q}}+1}\right).
\end{equation}

For $\texttt{T}_3$ one can apply (A\ref{ass:add} 4.)
\begin{equation}\label{eq:bias_bound}
\mathbb{E}[\texttt{T}_3^2] \leq C\Delta_{\texttt{L}}^2.
\end{equation}
The proof is now completed by noting \eqref{eq:master},  applying the $C_2-$inequality twice and using
\eqref{eq:t1_bound},  \eqref{eq:t2_bound} and \eqref{eq:bias_bound} which concludes the proof.
\end{proof}

\end{document}